\documentclass[11pt,oneside,reqno]{amsart}
\usepackage[english]{babel}
\usepackage{amsmath,amsfonts,amsthm,amssymb,amsbsy,upref,color,graphicx,amscd,hyperref,enumerate,bbm,comment, algorithm, algorithmic,dsfont}
\hypersetup{colorlinks,breaklinks,
            linkcolor=[rgb]{0,0,0},
            citecolor=[rgb]{0,0,0},
            urlcolor=[rgb]{0,0,0}}

\usepackage{xfrac,xcolor}
\usepackage{mathtools,url}
\usepackage[active]{srcltx}
\usepackage{cases} %for accolades with numbered equations
\usepackage[shortlabels]{enumitem}
\usepackage{cancel}

\usepackage[capitalise, nameinlink, noabbrev]{cleveref}

\def\eps{{\varepsilon}}

\def\O{\Omega}
\def\R{\mathbb{R}}

\newcommand{\be}{\begin{equation}}
\newcommand{\ee}{\end{equation}}

\newcommand{\loc}{\mathrm{loc}}

\theoremstyle{plain}

\newtheorem{theorem}{Theorem}[section]
\newtheorem{proposition}[theorem]{Proposition}
\newtheorem{corollary}[theorem]{Corollary}
\newtheorem{lemma}[theorem]{Lemma}
\newtheorem{definition}[theorem]{Definition}

\theoremstyle{remark}
\newtheorem{remark}[theorem]{Remark}

\numberwithin{equation}{section}

\def\XXint#1#2#3{{\setbox0=\hbox{$#1{#2#3}{\int}$ }
\vcenter{\hbox{$#2#3$ }}\kern-.6\wd0}}

\title[Some remarks on singular capillary cones with free boundary]{Some remarks on singular capillary cones\\with free boundary}
\author[A.~Pacati, G.~Tortone, B.~Velichkov]{Alberto Pacati, Giorgio Tortone, Bozhidar Velichkov}

\address{Alberto Pacati \newline \indent
 Department of Mathematics\newline\indent ETH Zurich,\newline
\indent
 Ramistrasse 101, 8092 Zurich, Switzerland}
 \email{alberto.pacati@math.ethz.ch}

\address{Giorgio Tortone\newline\indent Dipartimento di Matematica "G. Peano" \newline\indent Universit\`a degli Studi di Torino \newline\indent Via Carlo Alberto 10, 10124, Torino, Italy} \email{giorgio.tortone@unito.it}

\address {Bozhidar Velichkov \newline \indent
Dipartimento di Matematica \newline\indent Universit\`a di Pisa \newline \indent
Largo Bruno Pontecorvo, 5, 56127 Pisa, Italy}
\email{bozhidar.velichkov@unipi.it}

\begin{document}

\thanks{{\bf Acknowledgments.}
	G.T. and B.V. are supported by the European Research Council's (ERC) project no. 853404 ERC VaReg - \it Variational approach to the regularity of the free boundaries\rm, financed by the Horizon 2020 program.
	G.T. is a member of INdAM-GNAMPA}

% \begin{abstract}
% We study singular minimizing capillary cones with free boundary.  Precisely, we provide a stability criterion \emph{à la} Jerison-Savin for capillary hypersurfaces and show that, in dimensions up to $4$, minimizing cones with non-sign-changing mean curvature are flat.
% We apply this criterion to minimizing capillary drops and show, moreover, that nontrivial axially symmetric cones are unstable in dimensions up to $6$.

% The main results rely on a Simons-type inequality for convex, homogeneous, symmetric functions of the principal curvatures, combined with a boundary condition intrinsic to the capillary setting.
% \end{abstract}

\begin{abstract}
We study singular minimizing capillary cones with free boundary. We provide a stability criterion \emph{à la} Jerison--Savin and use it to prove that a minimizing capillary cone is flat when its free boundary mean curvature is non-negative and $n\leq 4$, or non-positive and $n\leq 6$.

We also show that minimizing cones with axially symmetric free boundary are flat in dimensions up to \(6\), and derive improved regularity results for graphical minimizing capillary hypersurfaces.

The proofs rely on Simons-type inequalities for convex, homogeneous, symmetric functions of the principal curvatures, coupled with a boundary identity intrinsic to the capillary setting.
\end{abstract}

\maketitle
\setcounter{tocdepth}{1}
%\tableofcontents
\section{Introduction}
 
Let $D\subset \R^{n+1}$ be a smooth domain, $\sigma \in C^1(\partial D;(-1,1))$ and $g\in C^1(\overline{D})$.  The capillarity problem with free boundary
consists of finding hypersurfaces $M\subset \overline{D}$ satisfying 
$$
H_M = g \quad\mbox{on}\quad M\cap D,\qquad 
\cos(\theta)=\sigma \quad\mbox{on}\quad  M\cap \partial D,
$$
where $H_M$ is the mean curvature of $M$ and $\theta$ is the contact angle between $M$ and the boundary of the container $\partial D$. From a variational perspective, such hypersurfaces arise as boundaries of local minimizers of the functional
\begin{equation}\label{e:functional-macro}
    \mathcal{F}(E):=\mathcal{H}^{n}(\partial E\cap D) + \int_{E} g\,dx -\int_{\partial E\cap \partial D}\sigma\,d\mathcal{H}^{n},
\end{equation}
among sets $E \subset \R^{n+1}$ of locally finite perimeter
in $D$ (see \cite{deMasi-dePhilippis,hong-saturnino,minmax} for a min-max construction). In analogy with the classical Gauss free energy model describing equilibrium states of incompressible fluids, the  functional consists of three contributions: the \emph{free surface energy}, 
proportional to the perimeter of the separating interface between the fluid and its surroundings, the \emph{gravitational energy} and the \emph{wetting free energy}, which accounts for the adhesion between the fluid and the container walls $\partial D$ (see \cite{finn,young} for a historical survey on the subject).

The existence, regularity, and geometric properties of capillary
surfaces have been actively studied in recent years with particular attention to their geometric applications (see the seminal papers \cite{almgren,DeG,gruter2,gruter-jost}, and more recent contributions \cite{dephimaggi,demasi-edelen-gasparetto-li,fraser-schoen,li,rossouam,wang-xia}). From the viewpoint of regularity, the problem has two distinct aspects.
On the one hand, the interior regularity of $M$ falls within the
well-developed theory of minimal and prescribed-mean-curvature
hypersurfaces. On the other hand, the interaction of $M$ with
$\partial D$ gives rise to genuinely free boundary phenomena, and
several fundamental questions remain open.\\

A particularly important special case is the non-parametric setting, in which the capillary hypersurface is represented as a graph over the supporting hyperplane. The problem then reduces to finding a function $u\colon \Omega\subset \R^n \to \R$ satisfying the following boundary value problem: 
$$
\mathrm{div}\left(\frac{\nabla u}{\sqrt{1+|\nabla u|^2}}\right) = g\quad\mbox{in }\Omega\cap\{u>0\},\qquad
|\nabla u| = \sqrt{\frac{1-\sigma^2}{\sigma^2}}\quad\mbox{on }\Omega\cap\partial\{u>0\}.
$$
The variational formulation of this problem was first introduced by Caffarelli and Friedman \cite{cf}. As they observed, the resulting model can be viewed as a quasilinear analogue of the classical one-phase Bernoulli problem (see also \cite{cm,karak}).  Specifically, by replacing the area functional with the Dirichlet energy, the minimization problem aligns with the one studied by Alt and Caffarelli \cite{altcaf} (see also \cite{cjk,desilva,dj,js,v,w}).

\subsection{Boundary regularity and the notion of free boundary}\label{sub:intro-topology} In what follows, we say that a set $E \subset \overline{D}$ is {\it open} if it is {\it relatively open in $\overline{D}$}. Moreover, we denote by $M:=\partial E$ {\it the boundary of $E$} with respect to the topology of $\overline{D}$ and by $\partial M:=M\cap \partial D$  the {\it free boundary of $M$}.\\ 
For instance, according to this convention, for any non-negative continuous function $u\colon \R^n \to \R$, we can define a graphical hypersurface $M$ as the boundary of the set 
$$E=\Big\{x\in\R^{n+1}\ :\ 0\le x_{n+1}<u(x_1,\dots,x_n)\Big\}.$$
In this case, the free boundary $\partial M$ is precisely the boundary of the support of $u$ in $\R^n$ and the hypersurface $M$ is said to be \emph{graphical}. Unlike the case of complete minimal hypersurfaces, the capillary framework allows for a choice of the plane over which the surface is graphical. For instance, surfaces graphical over a plane orthogonal to the constraint hyperplane are expected to be minimizing, while surfaces graphical over the constraint hyperplane itself relate to the one-phase free boundary problem, as in \cite{cel}.

In what follows, we set $\R^{n+1}_+$ to be the half-space $\{x_{n+1}>0\}$. Given $D \subset \R^{n+1}$, a set of locally finite perimeter $E\subset \R^{n+1}$ is said to be a minimizer of $\mathcal{F}$
in $U\subset \R^{n+1}$ if
$$
\mathcal{F}(E\cap U)\leq\mathcal{F}(\widetilde{E}\cap U),
$$
among all sets of locally finite perimeter $\widetilde{E}\subset \R^{n+1}$ satisfying $E\Delta \widetilde{E}\subset \subset U$. In such a case, the boundary $M:=\partial E$ is said to be a \emph{minimizing capillary surface}. For the study of singular points, it is natural to introduce the concept of capillary cones. Precisely, given $\theta \in (0,\pi)$ we set
\begin{equation}\label{e:functional-blow}
    \mathcal{F}_\theta(E):=\mathcal{H}^{n}(\partial E\cap \R^{n+1}_+)-\cos(\theta)\mathcal{H}^{n}(\partial E\cap \{x_{n+1}=0\}),
\end{equation}
for every set of locally finite perimeter $E\subset \R^{n+1}$. Therefore, if $E$ is a homogeneous minimizer of  $\mathcal{F}_\theta$ in $\R^{n+1}$, we say that $M:=\partial E\subset \{x_{n+1}\geq 0\}$ is a \emph{minimizing capillary cone}.\\

Finally, we collect here the main known result regarding the regularity of minimizing capillary surfaces $M$, both in the interior of $D$ and at its free boundary $\partial M$.

\subsection{Regularity in the interior of $D$}
Let $M$ be a minimizing capillary surface in $D$. Then, the interior of $M$ in $D$ can be decomposed as the disjoint union of a regular part and a singular part:
	\begin{enumerate}
		\item[\rm (i)] $\mathrm{Reg}(M)$ is locally a smooth hypersurface 
		with prescribed mean curvature, namely
  \begin{equation}\label{e:criticality-M}
    H_M=g\qquad \mbox{on}\quad M\cap D;
\end{equation}
		\item[\rm (ii)] the singular set $\mathrm{Sing}(M)$ is relatively closed and $
\mathrm{dim}_{\mathcal{H}}(\mathrm{Sing}(M))\leq n-7$.\vspace{0.1cm}
	\end{enumerate}
The interior regularity of \(M\) follows from the classical theory of area-minimizing hypersurfaces; see the monographs \cite{maggi-lecture,simon-lecture}. We stress that the Hausdorff estimate on the singular set is sharp. Indeed, the argument of Schoen-Simon-Yau
\cite{schoen-simon-yau} excludes cones with isolated singularities in dimensions up to $7$, whereas sharpness follows from the minimality of the Simons' cone \cite{bombieri-degiorgi-giusti}.

\subsection{Regularity of the free boundary $\partial M$}
Let us introduce some notation related to the vector frame associated to regular points of $M$ and the free boundary $\partial M$: 
\begin{enumerate}
\item[\rm (i)] $\nu$ is the unit normal vector field
of $M$ that points out of $E$; 
\item[\rm (ii)] $\eta$ is the unit normal vector field of $\partial M$ in $M$ that points
out of $\overline{D}$;
\item[\rm (iii)] $\overline{\nu}$ is
the unit normal vector field of $\partial M$ in $\partial D$ that points out of $E \cap \partial D$.
\end{enumerate}
Thus, we can decompose the free boundary $\partial M$ into a disjoint union of a (relatively open) regular part $\,\mathrm{Reg}(\partial M)$ and a singular part $\,\mathrm{Sing}(\partial M)$ such that:\vspace{0.1cm}
	\begin{enumerate}
		\item[\rm (i)] in a neighborhood of  $\mathrm{Reg}(\partial M)$, the boundary $M$ is an $n$-dimensional $C^{1,\alpha}$ manifold with boundary that meets $\partial D$ at the prescribed contact angle. Indeed, 
\be\label{e:criticality-partialM}
\cos(\theta) = \sigma\qquad\mbox{on}\quad\mathrm{Reg}(\partial M),
\ee
where $\theta$ is the angle between the normal $\nu$ and the inward normal vector of $\partial D$. Moreover, $\mathrm{Reg}(\partial M)$ is locally a $(n-1)$-dimensional $C^{1,\alpha}$ embedded submanifold of $\partial D$;\\
		\item[\rm (ii)] 
the singular set $\mathrm{Sing}(\partial M)$ is relatively closed and there exists a critical dimension $n^*=n^*(\theta)$,
possibly depending on the value of $\theta$,  such that
$$
\mathrm{dim}_{\mathcal{H}}(\mathrm{Sing}(\partial M))\leq n-n^*(\theta). \vspace{0.1cm}
$$
	\end{enumerate}
The study of the regular set follows from an $\eps$-regularity theorem (see \cite{taylor} for $n=2$, \cite{dephimaggi} for $n>3$ and \cite{demasi-edelen-gasparetto-li} in the framework of varifolds). By Federer's dimension-reduction argument, it is possible to characterize the threshold $n^*(\theta)$ as the smallest dimension at which minimizing capillary cones in $\{x_{n+1}\geq 0\}$ exhibit isolated singularities. Naturally, below this dimensional threshold, the only minimizing cone is the flat one.\\ 

Taylor \cite{taylor} first estimated the dimension of the singular set of $\partial M$ by showing that $n^*(\theta)\geq 3$ for every contact angle (see \cite{dephimaggi} for the anisotropic analogue). On the other hand, in the special case $\theta=\frac{\pi}{2}$ (i.e. free-boundary minimal surface), Gr\"{u}ter \cite{gruter} proved the sharp estimate $n^*(\frac{\pi}{2}) = 7$.

More recently, Chodosh, Edelen, and Li \cite{cel} provided a first improvement of this bound by analyzing minimizing capillary cones in the regimes in which the contact angle approaches $0$ (i.e. close to the one-phase Bernoulli free boundary problem), or $\frac{\pi}{2}$ (i.e. close to free-boundary minimal surfaces). Indeed, they prove that $n^*(\theta)\geq 4$, for every contact angle, and that the following
rigidity result for minimizing cones holds true: there exists a universal $\eps>0$ such that if 
\be\label{chodosh-main}
n=4, \quad\mbox{for }0<\theta<\eps\mbox{ or }0<\pi-\theta<\eps ;\qquad
n<7, \quad\mbox{for }\left|\theta-\frac{\pi}{2}\right|\leq \eps,
\ee
then $M$ is flat. We stress that the study for angles close to $\frac{\pi}{2}$ 
was initiated by Li-Zhou-Zhu \cite{minmax}.\\

The main objective of this paper is to sharpen the Hausdorff-dimension estimates for the singular set of the free boundary
of minimizing capillary hypersurfaces, without imposing any restriction on the value of the contact angle. A second objective is to clarify the close relationship between the capillarity problem and the one-phase Bernoulli problem, particularly at the level of their stability inequalities.

\subsection{Main results}\label{sub:intro-main} The main result of the present paper is the following.

\begin{theorem}\label{thm-main}
Let $M$ be a minimizing capillary cone. Then, the following holds: 
\begin{enumerate}
    \item if $H_{\partial M}\geq 0$ and $n\leq 4$, then $M$ is flat;\vspace{0.05cm}
    \item if $H_{\partial M}\leq 0$ and $n\leq 6$, then $M$ is flat.
\end{enumerate}
\end{theorem}

By using the same argument, in Section \ref{s:cel} we also give an alternative proof of the non-existence of singular minimizing cones in dimension $n=3$ (precisely point (1) of \cite[Theorem 1.2]{cel}). We emphasize that, in this case, no assumption on the mean curvature of $\partial M$ is needed.

These results suggest that every minimizing capillary cone should be flat in dimensions $n\leq6$, without any sign assumption on
$H_{\partial M}$. 
As supporting evidence, we prove this conjecture under the additional assumption that the free boundary is axially
symmetric.

\begin{theorem}\label{thm-axially}
Let $M$ be a minimizing capillary cone with $n\leq 6$. If the free boundary $\partial M$ is axially symmetric, then $M$ is flat. 
\end{theorem}

Nevertheless, Theorem \ref{thm-main} finds an immediate application in the case of graphical cones, where we can prove that the mean curvature $H_{\partial M}$ is non-negative. Hence, set $n^*_G(\theta)$ to be the smallest dimension at which minimizing capillary graphical cones exhibit isolated
singularities.

\begin{theorem}\label{thm-main-graphical-cone}
 Let $M$ be a minimizing capillary graphical cone with $n=4$. Then, $M$ is flat and $$
 n^*_G(\theta) \geq 5,\qquad\mbox{for every }\theta \in \left(0,\frac{\pi}{2}\right).
 $$
 \end{theorem}

Since such configurations arise as blow-up limits in the analysis of minimizing capillary drops as in \cite{cf, cm}, the following improved estimate is a direct consequence of Federer’s dimension-reduction and the rigidity result for minimizing capillary graphical cones in low dimensions. 
\begin{corollary}\label{thm-main-graphical}
Let $\Omega\subset \R^{n}$ be bounded, $\sigma \in C^1(\Omega;(-1,1)), g\in C^1(\Omega)$ and $u \in H^1(\Omega)$ be a local minimizer of the functional
$$
\mathcal{J}(u):=\int_{\Omega\cap \{u>0\}} \sqrt{1+|\nabla u|^2} + gu  - \sigma \,dx .
$$
Then $\mathrm{dim}_{\mathcal{H}}(\mathrm{Sing}(\partial \{u>0\} ))\leq n-5$.  
\end{corollary}
We note that, as it was shown in \cite{cf}, the sets of finite perimeter $E$ that minimize the functional \eqref{e:functional-macro} in the half-space $D=\{x_{n+1}>0\}$ are graphical, so Corollary \ref{thm-main-graphical} applies, giving the smoothness of the free boundary in dimensions up to $4$. 
\subsection{Strategy and the role of  
\texorpdfstring{$H_{\partial M}$}{Lg}}
The proof is inspired by the rigidity argument of Jerison and Savin
\cite{js} for the one-phase Bernoulli problem, and by the classical
approach to stable minimal cones \cite{schoen-simon-yau,simon-lecture,simons}. Thanks to Federer's dimension-reduction argument, we may restrict our attention to cones $M$ having at most an isolated singularity at the origin.
The core of the argument is to construct a homogeneous function $w$, built from the principal curvatures of $M$, that satisfies both an interior Simons-type inequality and a compatible boundary inequality
\be\label{e:linearized-system-intro}
\begin{aligned}
\displaystyle \frac12\Delta_M w^2 -|\nabla_M w|^2 + |A|^2 w^2 &=0 &&\quad\mbox{in }M,\\
\cos(\theta)H_{\partial M} w^2 -  \frac12(\nabla_M w^2\cdot \eta) &= 0 &&\quad\mbox{on }\partial M\setminus \{0\}.
\end{aligned}
\ee
The stability criterion in Proposition \ref{p:JS} shows that the existence of a strict subsolution of this type forces $w\equiv0$, and hence the cone to be flat.
In the spirit of Simons \cite{simons}, it is natural to consider the squared norm of the second fundamental form as a competitor in the stability inequality. While this choice allows us to rule out singularities in dimension $n=3$ and the existence of nontrivial axially symmetric free boundaries in dimensions up to $6$, the general case requires a more refined construction (see Section \ref{s:finally}).

The main difficulty lies in the boundary inequality. The competitor
must weight differently the principal curvatures having the same sign as $H_{\partial M}$ and those having the opposite sign. This leads to the following choices. 
\begin{enumerate}
    \item if $H_{\partial M}\geq 0$ and $n=4$, we choose   $w:=F(A)^{1/3}$ with
 $$
F(A):=\sin(\theta)\Bigg(\sum_{\kappa_i\geq 0}\kappa_i^2 + 4 \sum_{\kappa_s<0}\kappa_s^2 + 4 H_{\partial M}^2\Bigg)^{1/2}\quad \text{on }\partial M \setminus \{0\},
$$
\item if $H_{\partial M}\leq 0$ and $n=4,5,6$, we choose $w:=F(A)^{\frac{n-2}{n-1}}$ with
  $$
F(A):=\sin(\theta)\Bigg(\sum_{\kappa_i\geq 0}\kappa_i^2 +  (n-2)^2\sum_{\kappa_s<0}\kappa_s^2 +  H_{\partial M}^2\Bigg)^{1/2}\quad \text{on }\partial M \setminus \{0\},
$$
\end{enumerate}
where $\kappa_j$ are the principal curvatures of $\partial M$. As highlighted in Remark \ref{rmk:competitors}, these new competitors are optimal within the class of quadratic functions of the principal curvatures of $M$.\\

For general minimizing cones, the lack of graphicality prevents us from controlling the sign of $H_{\partial M}$, which accounts for the distinction between the two cases in Theorem \ref{thm-main}. This issue is specific to the general, non-graphical setting in dimensions higher than $3$. While graphical singular cones are known to have positive free boundary mean curvature, for general cones it is not even known whether $H_{\partial M}$ can vanish. 

A striking feature of our result is that, in the regime $H_{\partial M}\leq0$, singularities can
nevertheless be excluded up to dimension $6$, as in the
free-boundary minimal surface case.

Nevertheless, our result completes the analysis of capillary graphical cones in dimension $4$ started in \cite{cel}, where the authors derive the estimate on the singular set from the fact that capillary cones are graphical for small contact angles.

\subsection{Connection with the one-phase Bernoulli problem}
 Although the connection between the two problems has long been
recognized, only recently has it been systematically exploited in the study of capillary hypersurfaces (see \cite{cel, dephi-fusco-morini, demasi-edelen-gasparetto-li}). Our analysis further clarifies this relationship at the level of
stability and second variations. More precisely:
 \begin{itemize}
 \item[(i)] in Remark \ref{remark1}, we compare the notion of stability introduced by Caffarelli, Jerison and Kenig \cite{cjk} with the one introduced in Definition \ref{d:stable};
 \item[(ii)] in Remark \ref{remark2}, we highlight how the second fundamental form $A$ plays the role of the Hessian in \cite{js}.  
 \end{itemize}
In fact, this connection persists even at the level of the second-order expansion of the functionals. 

\subsection{Plan of the paper} 
In Section \ref{s:notations}, we introduce the coordinate system used throughout the paper and formulate an alternative notion of stability for capillary hypersurfaces with free boundary. The section concludes with Proposition \ref{p:JS}, which provides a  stability criterion \emph{à la} Jerison-Savin.
In Section \ref{s:A}, we use powers of the norm of the second
fundamental form as competitors in the stability inequality. This
yields Theorem \ref{thm-axially} and an alternative proof of
$n^*(\theta)\geq4$; see Section \ref{s:cel}. In Section \ref{s:finally}, we establish the interior and boundary
inequalities for homogeneous convex functions of the eigenvalues of
the second fundamental form and prove Theorems
\ref{thm-main} and \ref{thm-main-graphical-cone}, together with
Corollary \ref{thm-main-graphical}. Appendix \ref{a:eigen} collects the required computations for symmetric functions of eigenvalues.

\section{Stable capillary cones}\label{s:notations}
In this section we recall some facts about the stability of capillary cones. In particular, we introduce a system of coordinates that will be used throughout the paper. Finally, we derive a stability condition \emph{à la} Jerison-Savin (see \cite{js}).

\subsection{Set of coordinates}\label{s:coords} Let $M$ be a minimizing capillary cone with an isolated singularity at the origin and $x_0 \in \partial M\setminus \{0\}$. By the symmetry of the functional, we can restrict our analysis to the case $\theta \in (0,\pi/2)$. Then, we start by choosing coordinate axes in such a way that
$$
\overline\nu=e_n\ ,\qquad \nu= \sin(\theta)e_n + \cos(\theta)e_{n+1},\qquad \eta=\cos(\theta)e_n - \sin(\theta)e_{n+1}\qquad\mbox{at }x_0,
$$
where $\theta \in (0,\pi/2)$ defines the angle between $\nu$ and $e_{n+1}$ at a regular free boundary point.\\
Assume that in a neighborhood $\mathcal{N}_{x_0}$ of $x_0$, the boundary $M$ is given by the graph of a non-negative  function $u\colon \R^{n}\to \R$, that is
\begin{align*}
E\cap \mathcal{N}_{x_0}&=\{x\in \R^{n+1}\colon 0\leq x_{n+1}<u(x_1,\dots,x_n)\},\\ M\cap \mathcal{N}_{x_0} &=\{x \in \R^{n+1}\colon x_{n+1}=u(x_1,\dots,x_n)\},
\end{align*}
and so $\partial M \cap \mathcal{N}_{x_0} = \{x \in \R^{n+1}\colon x_{n+1}=0,u(x_1,\dots,x_n)=0\}$. Then, by rephrasing the criticality conditions \eqref{e:criticality-M}-\eqref{e:criticality-partialM} in terms of the local parametrization $u$, we get that
\be\label{e:criticality-u}
\mathrm{div}\left(\frac{\nabla u}{\sqrt{1+|\nabla u|^2}}\right)=0 \quad \mbox{in }\pi(M\cap \mathcal{N}_{x_0}),\qquad
|\nabla u|=\tan(\theta)\quad \mbox{on } \partial M \cap \mathcal{N}_{x_0},
\ee
where $\pi\colon \R^{n+1}\to \{x_{n+1}=0\}$ is the orthogonal projection on $\{x_{n+1}=0\}$. Similarly, writing $\partial M\cap \mathcal{N}_{x_0}$ as the graph of a regular function $\varphi \colon \R^{n-1}\to \R$ such that $\nabla^2 \varphi(x_0)$ is diagonal and differentiating the condition
$$
u(x',\varphi(x'))=0 \qquad \mbox{at } \varphi^{-1}(x_0 \cdot e_{n+1}),
$$
with respect to tangential directions yields the following identities at $x_0 \in \partial M\setminus \{0\}$: 
\be\label{e:diffe0}
u_i =0\quad\mbox{for }i<n,\qquad u_n=-\tan(\theta),\qquad u_{ij}=u_{ijn}=0\quad\mbox{for } i\neq j,\mbox{ with }i,j<n.
\ee
Direct differentiation of the second equation in \eqref{e:criticality-u} implies that, for every $i<n$
$$
u_{in}=0 \quad\mbox{at }x_0,
$$
that is $\nabla^2 u(x_0)$ is diagonal. Ultimately, by differentiating again both equations in \eqref{e:criticality-u}, we get for every $i<n$ that
\begin{align}\label{e:diffe1}
\begin{aligned}
\sum_{j<n}u_{jj}&=-\frac{u_{nn}}{(1+u_n^2)},\qquad \sum_{j<n}u_{ijj}=-\frac{u_{inn}}{(1+u_n^2)},\\[0.3em]
u_{iin} &= \frac{1}{u_n}u_{nn}u_{ii} -\frac{1}{u_n}u_{ii}^2,\\[0.5em]
u_{nnn} &= \frac{1+u_{n}^2}{u_n}\sum_{j<n} u_{jj}^2 +\frac{1+3u_{n}^2}{u_n(1+u_{n}^2)}u_{nn}^2 \quad\mbox{at }x_0.
\end{aligned}
\end{align}
Finally, in a neighborhood of a regular free boundary point $x_0\in \partial M\cap \mathcal{N}_{x_0}$, we can rewrite the second fundamental form of $M$ in terms of the local parametrization $u$.

More precisely, in this coordinate system, the first fundamental form of $M$ takes the form
\be \label{e:tensor-g}
g := \mathrm{Id} + \nabla u \otimes \nabla u, \qquad g^{-1} = \mathrm{Id} - \frac{\nabla u \otimes \nabla u}{1+|\nabla u|^2}\qquad\mbox{in }\pi(M\cap \mathcal{N}_{x_0})
\ee
and so, at $x_0$
we can express the second fundamental form of $M$ as follows:
\begin{align*}
A:= &\,\frac{g^{-1}\nabla^2u}{(1+|\nabla u|^2)^{1/2}} = \frac{1}{(1+|\nabla u|^2)^{1/2}}\left(\nabla^2 u - \frac12\frac{\nabla u \otimes \nabla^2 u \nabla u}{1+|\nabla u|^2}- \frac12\frac{\nabla^2 u \nabla u \otimes \nabla u}{1+|\nabla u|^2}\right),\\[0.5em]
|A|^2 = &\,\frac{1}{1+|\nabla u|^2}\left(|\nabla^2 u|^2  - 2\frac{|\nabla^2 u \nabla u|^2}{1+|\nabla u|^2} + \frac{(\nabla^2 u\nabla u\cdot \nabla u)^2}{(1+|\nabla u|^2)^2} \right).
\end{align*}
We stress that $A$ is defined as the second fundamental form of $M$ induced by the outward normal $\nu$. Moreover, in this coordinate system, this yields to
\be\label{e:lambda}
\lambda_j(A)=\frac{u_{jj}}{(1+u_n^2)^{1/2}}\quad\mbox{for }j<n,\qquad \lambda_n(A)=\frac{u_{nn}}{(1+u_n^2)^{3/2}} \qquad\mbox{at }x_0.
\ee

The following lemma provides a formulation of the free boundary term in the stability condition associated to the functional \eqref{e:functional-blow}. Although the result is already known (see, e.g., \cite[Section 2.1]{hong-saturnino}), we provide the proof here for the sake of completeness. 
\begin{lemma}\label{l:H}
Let $E$ be a stationary point of $\mathcal{F}$ and $M:=\partial E$. Assume that the origin is an isolated singularity of $\partial M$. Then
$$
\cot(\theta)(A\eta\cdot\eta) = -\cos(\theta) H_{\partial M}
\qquad\mbox{at }\partial M\setminus \{0\},
$$
where $H_{\partial M}$ is the mean curvature of $\partial M$ pointing towards the complement of $\partial E \cap \{x_{n+1}=0\}$ in $\{x_{n+1}=0\}.$ Ultimately, we infer that 
$$
\lambda_n(A)=-\sin(\theta) H_{\partial M}\quad \mbox{at }\partial M\setminus \{0\}.
$$
\end{lemma}
\begin{proof}
Let $x_0 \in \partial  M\setminus \{0\}$ and let $\mathcal{N}_{x_0}$ be as before. Then, writing $\partial M\cap \mathcal{N}_{x_0}$ as the graph of a regular function $\varphi \colon \R^{n-1}\to \R$, by differentiating twice the identity $u(x',\varphi(x'))=0$ at $x_0$ we get that 
$$
u_n\varphi_{ii} + u_{ii}=0 \quad \mbox{at }x_0,
$$
and so
\be\label{e:H}
 H_{\partial M} = \sum_{i<n} \varphi_{ii} = -\frac{1}{u_n}\sum_{i<n}u_{ii} = \frac{1}{u_n(1+u_n^2)}u_{nn}\quad\mbox{at }x_0.
\ee
Here, in the last equality, we use \eqref{e:diffe1}. Then, since
$$
(A\eta\cdot\eta) =  \frac{u_{nn}}{(1+u_n^2)^{3/2}}\qquad\mbox{ at }x_0,
$$
the claim follows by substituting the second equation in \eqref{e:diffe0}. The identity for $\lambda_n(A)$ follows by direct substitution. 
\end{proof}
Notice that if the capillary surface $M$ is graphical in the $e_{n+1}$-direction, the previous lemma is the analog of \cite[Remark 1]{cjk} for the one-phase Bernoulli problem.

By combining Lemma \ref{l:H} with the stability condition associated to $\mathcal{F}$ (see \cite[Section 1]{rossouam} for the derivation), we can give the following definition of a \emph{stable capillary surface}.
\begin{definition}\label{d:stable}
Let $E\subset \R^{n+1}$. Then $M:=\partial E$ is said to be a \emph{stable capillary surface} with an isolated singularity at the origin if, given $\sigma=\cos(\theta)$, we have
$$
H_M=0 \quad\mbox{on }M\cap \{x_{n+1}>0\},\qquad (\nu\cdot e_{n+1})=\cos(\theta) \quad\mbox{on }\partial M\setminus \{0\},
$$
and for every function $\varphi\in C^\infty_c(\R^{n+1}\setminus\{0\})$ we have
   \begin{equation}\label{e:stability2}
    \int_M |\nabla_M\varphi|^2-|A|^2\varphi^2\,d\mathcal{H}^n\geq \cos(\theta)\int_{\partial M} H_{\partial M}\varphi^2\,d\mathcal{H}^{n-1}.
\end{equation}
Clearly, $M$ is said to be a \emph{stable capillary cone} if $M$ is also homogeneous. 
\end{definition}
\begin{remark}[The Caffarelli-Jerison-Kenig stability inequality]\label{remark1}
By rewriting the boundary term of the stability inequality in terms of the mean curvature of the free boundary $\partial M$, we can further deepen the connection with the one-phase Bernoulli free boundary problem. Indeed, Caffarelli-Jerison-Kenig \cite{cjk} proved that, given $\Lambda>0$, a stable solution $u\colon \R^n \to \R$ of the one-phase problem with an isolated singularity at the origin satisfies 
$$
\Delta u=0 \qquad \mbox{in }\{u>0\},\qquad |\nabla u|=\Lambda\quad \mbox{on }\partial\{u>0\}\setminus \{0\},
$$
and for every $\varphi \in C^\infty_c(\R^n \setminus \{0\})$, one has
$$
\int_{\{u>0\}}|\nabla \varphi|^2 \,dx \geq \int_{\partial\{u>0\}}H_{\partial\{u>0\}}\varphi^2 \,d\mathcal{H}^{n-1}.
$$
It is worth noting that the only significant difference between the notions of stability in the two problems lies in the sensitivity of \eqref{e:stability2} with respect to the contact angle.
\end{remark}
\subsection{Stability criterion \emph{à la} Jerison-Savin}
Stability conditions are the primary tool to prove rigidity results for PDEs. Precisely, in the same spirit of \cite{js, schoen-simon-yau}, in Proposition \ref{p:JS} we exploit \eqref{e:stability2} to prove a non-existence result for a nontrivial strict subsolution to the linearized equation of $\mathcal{F}$ in singular cones.
\begin{proposition}\label{p:JS} 
Let $M:=\partial E$ be a stable capillary cone with an isolated singularity at the origin. Suppose that for some $k\in \R$, there exists a $k$-homogeneous non-negative function $w \in L^\infty_\loc(\R^{n+1}\setminus\{0\})\cap C^\infty(\{w>0\}\setminus\{0\})$ satisfying the following differential inequalities in the distributional sense:
\be\label{e:assumption-js}
\begin{aligned}
\displaystyle \frac12\Delta_M w^2 -|\nabla_M w|^2 + |A|^2 w^2 &\geq \frac{\Lambda}{|x|^2}w^2 &&\quad\mbox{in }M,\\
\cos(\theta)H_{\partial M} w^2 -  \frac12(\nabla_M w^2\cdot \eta) &\geq 0 &&\quad\mbox{on }\partial M\setminus \{0\};
\end{aligned}
\ee
with 
\be\label{e:check}
\Lambda \geq \left(\frac{n}{2}+k-1\right)^2.\ee
Then, if one of the three inequalities above is strict,
we get $w\equiv 0$.
\end{proposition}
\begin{proof}
We just sketch the proof since it resembles that of \cite[Section 3]{schoen-simon-yau} and \cite{js} (see also \cite[Appendix B]{simon-lecture} and \cite{simons}). Suppose for contradiction that $w$ is nontrivial and consider the test function $\varphi:=w \phi$ with $\phi \in C^2_c(\R^{n+1}\setminus \{0\})$ to be chosen later. By integrating the stability inequality \eqref{e:stability2} by parts with this choice of test function, we obtain
\begin{align*}
\int_M w^2|\nabla_M \phi|^2 - \bigg(\frac12\Delta_M w^2 -|\nabla_M w|^2 &+ |A|^2 w^2\bigg)\phi^2  \,d\mathcal{H}^{n}\\
&\geq \int_{\partial M}  \left(\cos(\theta) H_{\partial M} w^2-\frac12(\nabla_M w^2\cdot \eta)\right)\phi^2\,d\mathcal{H}^{n-1}.
\end{align*}
Now, let $\Sigma:=M\cap\mathbb{S}^{n}$. First, by \eqref{e:assumption-js} and \eqref{e:check}, we know that the following defects
$$
\mathcal{D}_M:=
\frac12\Delta_M w^2-|\nabla_M w|^2+|A|^2w^2
-\frac{\Lambda}{|x|^2}w^2
\quad\mbox{and}\quad
\mathcal{D}_{\partial M}:=
\cos(\theta)H_{\partial M}w^2
-\frac12(\nabla_M w^2\cdot\eta)
$$
are non-negative. Moreover, since \(M\) is a cone and \(w\) is \(k\)-homogeneous,
\(\mathcal{D}_M\) and \(\mathcal{D}_{\partial M}\) are homogeneous
of degrees \(2k-2\) and \(2k-1\), respectively. Thus, for every radial test function
\(\phi(x):=\eta(|x|)\), with
\(\eta\in C^\infty_c(0,+\infty)\), the stability inequality  and polar
coordinates give
\[
\int_0^\infty r^{n+2k-1}(\eta'(r))^2\,dr
\geq
\widehat\Lambda
\int_0^\infty r^{n+2k-3}\eta(r)^2\,dr,
\]
where
\[
\widehat\Lambda
:=
\Lambda+ \frac{1}{\int_{\Sigma}w^2\,d\mathcal{H}^{n-1}}\left(\int_{\Sigma} \mathcal{D}_M \,d\mathcal{H}^{n-1} + \int_{\partial\Sigma} \mathcal{D}_{\partial M}\,d\mathcal{H}^{n-2}  \right).
\]
Here $\widehat\Lambda$ encodes the angular masses of the
two defects on the corresponding links. Since one of the three
inequalities in the statement is strict, we have
\[
\widehat\Lambda>
\left(\frac{n}{2}+k-1\right)^2
\]

On the other hand, the sharp weighted Hardy inequality gives
$$
\inf\left\{\frac{\int_0^\infty r^{n+2k-1}(\eta'(r))^2\,dr}{\int_0^\infty r^{n+2k-3} \eta(r)^2\,dr} : \eta \in C^\infty_c((0,+\infty))\right\} = \left(\frac{n}{2}+k-1\right)^2.
$$
This contradicts the previous inequality, hence $w\equiv0$ in $\R^{n+1}$.
\end{proof}
\section{The second fundamental form \texorpdfstring{$|A|$}{Lg}}\label{s:A}
In this section we start by showing that powers of the squared norm of the second fundamental form $A$ are admissible competitors for the stability criterion. 
Ultimately, we derive an alternative proof of the first point of \cite[Theorem 1.2]{cel} (i.e. $n^*(\theta)\geq 3$ for every contact angle) and we address the case of axially symmetric cones  (i.e. Theorem \ref{thm-axially}).

\subsection{Interior inequality}
We proceed by showing that powers $w:=|A|^\alpha$ satisfy a Simons-type inequality in cones $M$. We stress that, in the case of minimal surfaces, the original computation on the link was performed in \cite{simons}, while the extension to powers of $|A|$ was carried out in \cite{schoen-simon-yau}. Our analysis, however, requires a different inequality specifically tailored to handle the free boundary, which closely follows the one obtained in \cite{js} for the one-phase problem.

\begin{lemma}\label{l:alpha}
Let $E$ be a homogeneous stationary point of $\mathcal{F}$ and $M:=\partial E$. Then, given $\alpha \in (0,1)$, the function $w:=|A|^\alpha$ satisfies the following inequalities in a distributional sense:
$$
\frac12\Delta_M w^2 - |\nabla_M w|^2 + |A|^2 w^2 \geq \left( 
2\frac{n-2}{n-1}\frac{\alpha}{|x|^2} +\alpha\left(\alpha-1+\frac{2}{n-1}\right)\frac{|\nabla_M |A||^2}{|A|^2}\right)
w^2
\quad\mbox{on }M,
$$
and 
$$
\frac12\Delta_M w^2 - |\nabla_M w|^2 + |A|^2 w^2 \geq  \alpha(\alpha+1)\frac{w^2}{|x|^2}\quad\mbox{on }M,\qquad \mbox{if }\alpha\geq 1-\frac{2}{n-1}.
$$
\end{lemma}
\begin{proof}
Set $c:=|A|$. Then, by adapting the Simons' inequality (see Lemma \ref{l:generalized-simon} for a more general derivation),
we have
$$
\frac12\Delta_M c^2 - |\nabla_M c|^2 + |A|^2 c^2 \geq 
\frac{2}{n-1}|\nabla_M c|^2 + 2\frac{n-2}{n-1}\frac{c^2}{|x|^2} \qquad\mbox{on }M.
$$
Now, for $\eps>0$ we set $w_\eps := \left( c^2 + \eps\right)^{\alpha/2}$ with $\alpha \in (0,1)$, and
we obtain that on $M$
\begin{align*}
    \frac12 \Delta_M w_\eps^2 \,-\, &\,|\nabla_M w_\eps|^2 + |A|^2 w_\eps^2\\ 
    \geq & \, \left[\left(1-\alpha \frac{c^2}{c^2+\eps}\right)|A|^2 + \frac{n-2}{n-1}\frac{2\alpha}{|x|^2}\left(\frac{c^2}{c^2+\eps}\right) 
     + \alpha \frac{|\nabla_M c|^2}{c^2+\eps}\left(\frac{n+1}{n-1}-(2-\alpha)\frac{c^2}{c^2+\eps}\right)\right]w_\eps^2.
\end{align*}
Then, if we set $w:=|A|^\alpha$, by letting $\eps \to 0^+$,  
we infer that
\begin{align*}
&\frac12\Delta_M w^2 -|\nabla_M w|^2 + |A|^2 w^2\\ 
& \qquad\qquad\geq \left((1-\alpha) |A|^2 + 
2\frac{n-2}{n-1}\frac{\alpha}{|x|^2} +\alpha\left(\alpha-1+\frac{2}{n-1}\right)\frac{|\nabla_M c|^2}{c^2}\right)w^2
\qquad\mbox{on }M,
\end{align*}
in a distributional sense.
While the first inequality follows from the fact that $\alpha<1$, the second inequality holds because, if $M$ is a cone, the function $c$ is homogeneous and 
$$
|\nabla_M c|^2 \geq \frac{1}{|x|^2}c^2\quad\mbox{on }M.
$$
\end{proof}
\begin{remark}
We emphasize that throughout the paper, all differential equations can be rigorously derived by regularizing the competitors near their nodal set. This type of approximation is well known in the literature, including in the case of Simons' identity (see also the derivation in \cite{schoen-simon-yau}), and is therefore often omitted.
\end{remark}
\subsection{Boundary inequality}
We proceed by computing the normal derivative of the squared norm of the second fundamental form on the free boundary $\partial M$. Although in \cite[Lemma C.2]{minmax} the authors present the same boundary condition, we provide a different proof here for the sake of clarity (see Lemma \ref{l:generalized-simon-boundary} for a further generalization). 
\begin{lemma}\label{l:boundary}
Let $E$ be a homogeneous stationary point of $\mathcal{F}$ and $M:=\partial E$. Assume that $\partial M$ has an isolated singularity at the origin. Then
    \be\label{e:boundary-eq}
    \frac12 (\nabla_M |A|^2\cdot\eta) = 2\cos(\theta)H_{\partial M}|A|^2 +\cot(\theta)\sum_{j=1}^n \lambda_j^3(A)
\quad\mbox{on }\partial M \setminus \{0\},
    \ee
    where $\lambda_j(A)$, with $j=1,\dots,n$, are the eigenvalues of the second fundamental form $A$.
\end{lemma}
\begin{proof}
For the sake of simplicity, all the estimates and identities in this subsection are valid at a fixed point $x_0\in \partial M\setminus \{0\}$. Therefore, as in Section \ref{s:coords}, after fixing $x_0$, we can consider $M\cap \mathcal{N}_{x_0}$ as the graph of a regular function $u\colon \R^n \to \R$. Direct differentiation yields  
\begin{align*}
\frac12 \partial_n |A|^2 =& -\frac{u_{nn}u_n}{(1+u_n^2)^2}\left(\sum_{i<n} u_{ii}^2 + \frac{u_{nn}^2}{(1+u_n^2)^2}\right)\\
& +\frac{1}{1+u_n^2}\left(\sum_{i<n}u_{ii}u_{iin} + u_{nn}u_{nnn} + \frac{u_n^3 u_{nn}^3}{(1+u_n^2)^2} +
 \frac{u_n^4 u_{nn} u_{nnn} + u_n^3 u_{nn}^3}{(1+u_n^2)^2} - \frac{2 u_n^5 u_{nn}^3 }{(1+u_n^2)^3}\right)\\
 & - \frac{1}{1+u_n^2}\left(\frac{2 u_{nn}u_n^2 u_{nnn} + 2u_{nn}^3 u_n}{1+u_n^2}-\frac{2u_n^3 u_{nn}^3}{(1+u_n^2)^2}\right)\\
 =& - \frac{u_{nn}u_n}{(1+u_n^2)^2}\left(\sum_{i<n}u_{ii}^2 - \frac{1+u_n^2}{u_{nn}u_n}\sum_{i<n}u_{ii}u_{iin}\right) - \frac{3 u_n u_{nn}^3}{(1+u_n^2)^4} + \frac{u_{nn}u_{nnn}}{(1+u_n^2)^3}.
\end{align*}
Notice that
\begin{align*}
\left(\sum_{i<n} u_{ii}^2 -\frac{1+u_n^2}{u_n u_{nn}}\sum_{i<n}u_{ii}u_{iin} \right)&=
\left(\sum_{i<n} u_{ii}^2 - \frac{1+u_n^2}{u_n^2}\sum_{i<n} u_{ii}^2 + \frac{1+u_n^2}{u_n^2 u_{nn}} \sum_{i<n}u_{ii}^3\right)\\
&= -\frac{1}{u_n^2}\sum_{i<n}u_{ii}^2 + \frac{1+u_n^2}{u_n^2 u_{nn}}\sum_{i<n}u_{ii}^3,
\end{align*}
and
\begin{align*}
\frac{u_{nn}u_{nnn}}{(1+u_n^2)^3} &= \frac{1}{u_n(1+u_{n}^2)^2}u_{nn}\sum_{i<n} u_{ii}^2 +\frac{1+3u_{n}^2}{u_n(1+u_{n}^2)^4}u_{nn}^3.
\end{align*}
If we sum up the previous terms, it follows that
$$
\frac12 \partial_n |A|^2 = \frac{2 u_{nn}}{u_n (1+u_n^2)^{2}}\sum_{i<n}u_{ii}^2 + \frac{1}{u_n(1+u_n^2)^4}u_{nn}^3 - \frac{1}{u_n(1+u_n^2)}\sum_{i<n}u_{ii}^3,
$$
and since
\begin{align}\label{e:varie}
\begin{aligned}
\frac{1}{u_n}(1+u_n^2)^{1/2}\sum_{j=1}^n \lambda_j(A)^3 &= -\frac{1}{u_n(1+u_n^2)^{4}} u_{nn}^3 -\frac{1}{u_n(1+u_n^2)}\sum_{i<n} u_{ii}^3,\\
-\frac{2}{u_n}(1+u_n^2)^{1/2}|A|^2 (A\eta\cdot\eta) &= 2\frac{u_{nn}^3}{u_n(1+u_n^2)^4} + \frac{2u_{nn}}{u_n(1+u_n^2)^2}\sum_{i<n}u_{ii}^2,
\end{aligned}
\end{align}
it follows, by substituting \eqref{e:diffe0} and using \eqref{e:tensor-g}, that
\begin{align*}
\frac12 (\nabla_M |A|^2\cdot \eta) &= \cot(\theta)\left(-2 |A|^2 (A\eta\cdot\eta) + \sum_{j=1}^n \lambda_j^3(A) \right)\qquad\mbox{on }\partial M\setminus \{0\}.
\end{align*}
Finally, \eqref{e:boundary-eq} follows by applying Lemma \ref{l:H}.
\end{proof}
The following is a direct consequence of the previous lemma.
\begin{lemma}\label{l:boundary-L}
Let $E$ be a homogeneous stationary cone of $\mathcal{F}$, and let $M:=\partial E$ be its boundary with an isolated singularity at the origin. Consider $\alpha \in (0,1)$ so that the following holds 
\be\label{e:L.alpha3}
H_{\partial M}|A|^2 (1-2\alpha) - \alpha \frac{1}{\sin(\theta)}\sum_{j=1}^n \lambda_j^3(A) \geq  0\qquad\mbox{on }\partial M\setminus \{0\}.
\ee
Then $w:=|A|^\alpha$ satisfies the boundary inequality in \eqref{e:assumption-js} in a distributional sense.
Moreover, if $H_{\partial M} |A| \neq 0$ on $\partial M\setminus \{0\}$, by defining 
\be\label{e:L.alpha}
L(x):= 
2 + \frac{1}{\sin(\theta) H_{\partial M} |A|^2}\sum_{j=1}^n \lambda_j^3(A),
\ee
we can rewrite the condition \eqref{e:L.alpha3} as 
\be\label{e:L.alpha2}
H_{\partial M}(1-\alpha L)\geq 0\qquad\mbox{on }\partial M\setminus \{0\}.
\ee
\end{lemma}
\begin{proof}
Let $\eps>0$ and set $c:=|A|, w_\eps:=(c^2+\eps)^{\alpha/2}$. Then, by Lemma \ref{l:boundary}, we have that on $\partial M\setminus \{0\}$
\begin{align*}
\cos(\theta) H_{\partial M}w_\eps^2-\frac12(\nabla_M w^2_\eps\cdot\eta) &= w_\eps^2 \left[\cos(\theta) H_{\partial M}- \frac{\alpha}{2} \frac{(\nabla_M c^2\cdot \eta)}{c^2+\eps}\right]\\
& = w_\eps^2 H_{\partial M}\cos(\theta)\left(1-\frac{2\alpha c^2}{c^2+\eps}\right) - \cot(\theta)\alpha\frac{w_\eps^2}{c^2+\eps}\sum_{j=1}^n \lambda_j^3(A)
\end{align*}
Then, by letting $\eps \to 0^+$, we obtain that the following identity holds in a distributional sense:
$$
\cos(\theta) H_{\partial M}w^2-\frac12(\nabla_M w^2\cdot\eta) 
= \cos(\theta)\frac{w^2}{|A|^2}\left[H_{\partial M}|A|^2 (1-2\alpha) - \frac{\alpha}{\sin(\theta)}\sum_{j=1}^n \lambda_j^3(A) \right] \quad \mbox{on }\partial M\setminus \{0\}.
$$
Finally, by choosing $\alpha \in (0,1)$ so that \eqref{e:L.alpha3} holds, we deduce that the last right-hand side is non-negative on $\partial M$. On the other hand, \eqref{e:L.alpha2} follows by collecting $H_{\partial M}|A|^2$.
\end{proof}
\subsection{Axially symmetric free boundary} 
By assuming that all the principal curvatures of the free boundary are equal, we finally address the case of axial symmetry. Precisely, we deduce that, in dimensions up to $6$, stability implies that free boundaries are flat. 
\begin{proof}[Proof of Theorem \ref{thm-axially}]
Let $x_0 \in \partial M \setminus \{0\}$ and consider the set of coordinates introduced in Section \ref{s:coords}. Then, the principal curvatures $\kappa_j$ of $\partial M$ and $H_{\partial M}$ can be defined at $x_0$ as 
$$
\lambda_j = \sin(\theta)\kappa_j\quad\mbox{for }j=2,\dots,n-1,\qquad\mbox{while }\quad \lambda_n=-\sin(\theta) H_{\partial M}.
$$
(Clearly, if $e_1$ is the radial direction we have $\kappa_1=0$). Therefore, since all the principal curvatures are equal and their sum coincides with $H_{\partial M}$, it follows that
$$
\frac{1}{n-2} = \frac{\kappa_j}{H_{\partial M}}= -\frac{\lambda_j}{\lambda_n}\qquad\mbox{for every }j=2,\dots,n-1.
$$
By \eqref{e:L.alpha}, we first get $L\equiv (n-1)/(n-2)$ on $\partial M\cap \mathbb{S}^{n-1}$ and so we have two possibilities:
\begin{enumerate}
    \item[(\rm i)] $H_{\partial M}\geq 0$, and we can take  $\alpha \leq (n-2)/(n-1)$,
    \item[(\rm ii)] $H_{\partial M}<0,$ and we can take $\alpha \geq (n-2)/(n-1)$.
\end{enumerate}
Hence, by choosing $\alpha := (n-2)/(n-1)$, by Lemma \ref{l:alpha} and Lemma \ref{l:boundary-L}, it follows that $w:=|A|^\alpha$ satisfies in a distributional sense
\be\label{e:m}
\begin{aligned}
\displaystyle
\frac12\Delta_M w^2 -|\nabla_M w|^2 + |A|^2 w^2 &\geq \frac{(n-2)(2n-3)}{(n-1)^2}\frac{w^2}{|x|^2} &&\qquad\mbox{on }M,\\ \cos(\theta)H_{\partial M} w^2 -  \frac12(\nabla_M w^2\cdot \eta) &\geq 0 &&\qquad\mbox{on }\partial M\setminus \{0\};
\end{aligned}
\ee
In this case, the condition \eqref{e:check} can be rephrased as 
$$
\frac{(n-2)(2n-3)}{(n-1)^2}\geq  \left(\frac{n}{2}-\frac{n-2}{n-1}-1\right)^2 \quad\longleftrightarrow\quad 2\leq n\leq 6
$$
where the first inequality is strict for $2<n<6$. Therefore, if $n\in \{3,4,5\}$ we can directly apply Proposition \ref{p:JS} and infer that $w\equiv 0$, i.e. $M$ is flat.\\ 
Nevertheless, if $n=6$, we can show that the first inequality in \eqref{e:m} is strict, that is 
$$
\frac12\Delta_M w^2- |\nabla_M w|^2 + |A|^2 w^2 \geq \Lambda \frac{w^2}{|x|^2} \quad\mbox{on }M, \qquad \mbox{with }\Lambda > \frac{36}{25},
$$
which implies that condition \eqref{e:check} is strictly satisfied. Indeed, since
$$
\alpha = \frac45 > 1- \frac{2}{n-1} =\frac35,
$$
by Lemma \ref{l:alpha} we have
\begin{align*}
\frac12\Delta_M w^2 - |\nabla_M w|^2 +|A|^2 w^2 
&\geq \frac{32}{25}\frac{w^2}{|x|^2}  + \frac{4}{25}\frac{|\nabla_M |A||^2}{|A|^2} w^2 \geq \frac{36}{25}\frac{w^2}{|x|^2}  + \frac{4}{25}\frac{(\nabla_M |A|\cdot \xi)^2}{|A|^2} w^2\quad\mbox{on }M,
\end{align*}
for $\xi \in \mathbb{S}^{n-1}$ satisfying $\xi \cdot x=0$. On the other hand, by Lemma \ref{l:boundary}, we have
$$(\nabla_M(\log |A|)\cdot \eta)^2 > 0\qquad\mbox{on }\partial M.$$ 
Taking $\xi$ as an extension of $\eta$ in $M$, we deduce that the first inequality in \eqref{e:m} is strict in a neighborhood of every point of every point of $\partial M\setminus \{0\}$. Consequently, the strictness condition in Proposition \ref{p:JS} is satisfied, and we conclude that $w\equiv0$, hence $M$ is flat.

\end{proof}

\subsection{An alternative proof of $n^*(\theta)\geq 4$ }\label{s:cel}
We conclude the section by providing an alternative proof of $n^*(\theta)\geq 4$ (see point (1) of \cite[Theorem 1.2]{cel}). 
Thus, let $n=3$. Then, at every point $x_0 \in \partial M \setminus \{0\}$, we choose the coordinates in such a way that $x_0=x_0^1e_1$ where $e_1$ is the radial direction. Thus, we have $\lambda_1 =0$ and so, since $\mathrm{tr}(A)=0$, we infer that 
$$
\lambda_1(A)=0,\quad\lambda_2(A) = -\lambda_3(A)\qquad\mbox{at }x_0.
$$
Therefore, $\partial M$ is axially symmetric, and the result is a direct consequence of Theorem \ref{thm-axially}.

\section{The general case}\label{s:finally}

\subsection{Generalization of the Simons' inequality} By following the strategy in \cite[Section 4]{js}, we prove that the Simons' inequality is satisfied by any convex, symmetric and one-homogeneous function of the eigenvalues of the second fundamental form of $M$ (see Definition \ref{d:symmetric} for a precise notion of symmetric function).
\begin{lemma}\label{l:generalized-simon}
    Let $M$ be a minimizing cone, $A$ be its second fundamental form and $f\in C^{1,1}(\R^n; \R)$ be non-negative, symmetric, convex and homogeneous of degree one. Then, if we set $$c:=F(A)=f(\lambda_1(A),\dots,\lambda_n(A))$$ we have 
    \be\label{e:primo}
\frac12\Delta_M c^2 - |\nabla_M c|^2 + |A|^2 c^2 \geq 
\frac{2}{n-1}|\nabla_M c|^2 + 2\frac{n-2}{n-1}\frac{c^2}{|x|^2} \qquad\mbox{on }M.
    \ee
    Moreover, if $f$ is strictly convex and $c$ is non-constant, the previous inequality becomes strict.
\end{lemma}
\begin{proof}
Let $A=(a_{ij})_{ij} \in \R^{n(n+1)/2}$ and for every $x \in M$ we set
$$
c(x)=F(A(x))=f(\Lambda(x)), \quad \mbox{where}\quad \Lambda(x)=(\lambda_1(A(x)),\dots,\lambda_n(A(x)))
$$
is the vector of the $n$ real eigenvalues of $A$ (possibly equal) and $F \colon \R^{n(n+1)/2} \to \R$ is defined as in \eqref{e:definition-of-F}. Thus, after fixing a point $x_0:=(x^1_0,\dots,x^{n+1}_0) \in M$, it is not restrictive to assume, up to a rotation, that in a neighborhood $\mathcal{N}_{x_0}$ of $x_0$, the surface $M$ is written as the graph of a non-negative function $u\colon \R^n \to \R$ such that $$u(x_0)=|\nabla u(x_0)|=0\qquad\text{and}\qquad \nabla^2 u(x_0)\,\text{ is diagonal}.$$
With this local representation, we obtain that 
\be\label{e:last}
A= \nabla^2 u, \qquad H_M = \Delta u,\qquad 
\partial_k H_M 
= \sum_{i\neq k}u_{iik} + u_{kkk}\qquad \mbox{at }x_0 \in M.
\ee
For the sake of simplicity, throughout the proof we omit the dependence on the point $x_0$ and the indices of the summation run over $\{1,\dots,n\}$. Thus, in light of Proposition \ref{p:second-der-F}, by direct computation, the following estimates hold true at $x_0$
$$|\nabla_M c|^2 = \sum_{i,j,k}\left(f_{\lambda_i} f_{\lambda_j} \partial_k a_{ii}\partial_k a_{jj}\right)
$$
and
\be\label{e:stress}
\begin{aligned}
\frac12 \Delta_M c^2 &= f\sum_{i} f_{\lambda_i}\Delta_M a_{ii}  +\sum_{i,j,k} (f_{\lambda_i\lambda_j}f+f_{\lambda_i} f_{\lambda_j})\partial_{k}a_{ii}\partial_{k}a_{jj} 
+ \frac12\sum_{i<j}\sum_{k}(\partial_k a_{ij})^2 \partial^2_{a_{ij}}F^2(A) \\
&\geq \,f\sum_{i} f_{\lambda_i}  \Delta_M a_{ii} +|\nabla_M c|^2 
+\frac12 \sum_{i<j}\sum_{k}(\partial_k a_{ij})^2\partial^2_{a_{ij}}F^2(A) ,
\end{aligned}
\ee
where 
	\begin{align*}
	\partial_{a_{ij}}^2F^2(A)
 &=2 f\begin{cases}
	\displaystyle 2\frac{f_{\lambda_j}-f_{\lambda_i}}{\lambda_j-\lambda_i}&\,\text{if}\quad \lambda_i\neq\lambda_j\,,\medskip\\
2f_{\lambda_i\lambda_i}
- 2f_{\lambda_i\lambda_j}&\,\text{if}\quad \lambda_i=\lambda_j\,
	\end{cases}\\
  &=2 c\,\partial^2_{a_{ij}}F(A).
	\end{align*}
We stress that in the last line of \eqref{e:stress} we use the convexity and the positivity of $f$. Moreover, if $f$ is strictly convex and $c$ is non-constant, the inequality in \eqref{e:stress} becomes strict.\\
On the other hand, since at $x_0$ we have $\partial_{ii} H_M=0$, $\Delta_M a_{ii} = - \lambda_i |A|^2, \lambda_i = u_{ii},$ we deduce, in light of the homogeneity of $f$, that 
$$
\sum_{i}f_{\lambda_i}  \Delta_M a_{ii} = -|A|^2 \sum_i \lambda_i f_{\lambda_i}= -|A|^2 f \qquad\mbox{at }x_0.
$$
Therefore, by collecting the above identities and exploiting the local definition of the coefficient of $A$ (see the first equation in \eqref{e:last}), we obtain that
\be\label{e:remainder}
\frac12\Delta_M c^2 -|\nabla_M c|^2 +|A|^2 c^2 \geq \frac12 \sum_{i<j}\sum_{k}u_{ijk}^2 \partial^2_{a_{ij}}F^2(A) \qquad\mbox{at }x_0.
\ee
Now, since $\nabla^2 f$ is positive semi-definite and symmetric, we already know that $\partial^2_{a_{ij}}F^2(A)\geq 0$ if $\lambda_i=\lambda_j$. On the other hand, since $f$ is both symmetric and convex, we obtain
\be\label{e:ak}
(f_{\lambda_i} - f_{\lambda_j})(\lambda_i -\lambda_j ) \geq 0\qquad\mbox{for every }i\neq j,
\ee
which implies the non-negativity of the right-hand side in \eqref{e:remainder} (see also equation (4.3) in \cite{js}).
Notice that the last inequality is strict if $f$ is strictly convex and $\lambda_i\neq \lambda_j$. Therefore,
\be\label{e:s}
\frac12\Delta_M c^2 - |\nabla_M c|^2 +|A|^2 c^2 \geq c\sum_i \sum_{i\neq j}\partial^2_{a_{ij}}F(A) u_{iij}^2
\ee
where we keep track of the terms in which either $i=k$ or $j=k$. In order to obtain the first estimate, we proceed by considering the gradient of the competitor $c$: for every $k=1,\dots,n$ we know that 
$$
c_k = \sum_i f_{\lambda_i} \partial_k a_{ii} = \sum_i f_{\lambda_i} u_{iik}\quad \mbox{at }x_0.
$$
Nevertheless, since the mean curvature $H_M$ is constant, by \eqref{e:last} we infer that 
$$
\sum_{i\neq k} u_{iik} = -u_{kkk},\quad\mbox{which implies}\quad
c_k = \sum_{i\neq k}(f_{\lambda_i}-f_{\lambda_k})u_{iik} \quad \mbox{at }x_0.
$$
Moreover, since $f_{\lambda_i}- f_{\lambda_j} = \frac12 \partial^2_{a_{ij}}F(A)(\lambda_i - \lambda_j)$, by the Cauchy-Schwarz inequality we get
\be\label{e:kneq1}
c_k^2 \leq \left(\frac12\sum_{i\neq k}\partial^2_{a_{ik}} F(A) u_{iik}^2\right)\left(\sum_{i\neq k}(\lambda_i-\lambda_k)(f_{\lambda_i}-f_{\lambda_k})\right)\quad \mbox{at }x_0.	
\ee
Now, since $M$ is a smooth cone outside the origin, it is not restrictive to choose the coordinates at $x_0\neq 0$ in such a way that $x_0=x_0^1e_1$ where $e_1$ is the radial direction. Thus, for $i,j,k \neq 1$ we have
$$
\lambda_1 = 0,\quad u_{ij1}=-\frac{1}{|x|}u_{ij},\quad u_{11k}=0\quad\mbox{at }x_0.
$$
Thus, by considering first $c_1$, we get that at $x_0$
$$
c_1^2 \leq \frac12\left(\sum_{i\neq 1}\partial^2_{a_{i1}} F(A) u_{ii1}^2\right)\left(\sum_{i\neq 1}\lambda_i f_{\lambda_i}-f_{\lambda_1}\sum_{i\neq 1}\lambda_i\right) \leq \left(\sum_{i\neq 1}\partial^2_{a_{i1}} F(A) u_{ii1}^2\right)\frac{c}{2},
$$
where we use that $f$ is one-homogeneous and $\mathrm{tr}(A)=0$.\\ On the other hand, if we consider $k\neq 1$, we have that the following identities hold true at $x_0$:
\begin{align*}
\sum_{i\neq 1}\sum_{i<j}(\lambda_i-\lambda_j)(f_{\lambda_i}-f_{\lambda_j})& = \frac12 \left(
\sum_{i\neq 1} f_{\lambda_i} \sum_{j\neq \{1,i\}}(\lambda_i - \lambda_j) + \sum_{i\neq 1} \sum_{j\neq \{1,i\}} f_{\lambda_j} (\lambda_j - \lambda_i)\right)\\
& = (n-2) f - \sum_{i\neq 1}\lambda_i\sum_{j\neq \{1,i\}}f_{\lambda_j}\\
& = (n-2) f + \sum_{i\neq 1}\lambda_i f_{\lambda_i}\\
& = (n-1) f,
\end{align*}
where we use again the homogeneity of $f$ and $\mathrm{tr}(A)=0$. Thus, applying the previous identity in \eqref{e:kneq1} leads to 
\begin{align*}
c_k^2 &\leq \frac12 \left(\sum_{i\neq \{1,k\}}\partial^2_{a_{ik}} F(A) u_{iik}^2\right)\left(\sum_{i\neq \{1,k\}}(\lambda_i-\lambda_k)(f_{\lambda_i}-f_{\lambda_k})\right)\\
&\leq \frac12 \left(\sum_{i\neq \{1,k\}}\partial^2_{a_{ik}} F(A) u_{iik}^2\right)\left(\sum_{i\neq 1}\sum_{i<j}(\lambda_i-\lambda_j)(f_{\lambda_i}-f_{\lambda_j})\right)\\
& =
\frac{n-1}{2}\left(\sum_{i\neq \{1,k\}}\partial^2_{a_{ik}} F(A) u_{iik}^2\right) c.
\end{align*}
We stress that if $f$ is strictly convex, based on the observations related to \eqref{e:ak}, we have 
$$
\mbox{either}\quad c_k^2 < \frac{n-1}{2}\, c\! \sum_{i\neq \{1,k\}}\partial^2_{a_{ik}} F(A) u_{iik}^2 \quad\mbox{for some }k,\qquad\mbox{or}\qquad c\equiv 0.
$$
Indeed, if the first inequality were not strict for all $k$, it would follow that all the eigenvalues are equal, which would imply, in light of $\mathrm{tr}(A)=0$, that $c\equiv 0$. Finally, \eqref{e:primo} follows by summing the last inequality over the index $k$. Indeed, by exploiting the homogeneity of both $f$ and $u$, we get
$$
c_1=\sum\limits_i f_{\lambda_i}u_{ii1}=-\sum_i f_{\lambda_i}\frac{u_{ii}}{|x|}=-\frac{1}{|x|}\sum_i \lambda_i f_{\lambda_i}=-\frac{c}{|x|}
$$
and so
$$
\frac12\Delta_M c^2 - |\nabla_M c|^2 +|A|^2 c^2 \geq \frac{2}{n-1} \sum_{k\neq 1}c_k^2 + 2\frac{c^2}{|x|^2} =
\frac{2}{n-1} |\nabla_M c|^2 + 2\frac{n-2}{n-1}\frac{c^2}{|x|^2}\qquad\mbox{at }x_0.
$$
As we have already noted throughout the proof, if $c$ is non-constant, the strict convexity of $f$ implies the strict inequality in the latter case.
\end{proof}
By applying the approximation argument in the proof of Lemma \ref{l:alpha}, we can extend the class of admissible competitors as follows. 
\begin{lemma}\label{l:generalized-simon-2}
    Let $M$ be a minimizing cone, $A$ be its second fundamental form and $f\in C^{1,1}(\R^n; \R)$ be non-negative, symmetric, convex and homogeneous of degree one. Then, if we set $$c:=f(\lambda_1(A),\dots,\lambda_n(A)),$$ 
    for $\alpha \in (0,1)$, the function  $w:=c^\alpha$ satisfies 
    \be\label{e:alpha}
\frac12\Delta_M w^2 - |\nabla_M w|^2 + |A|^2 w^2 \geq \left( 
2\frac{n-2}{n-1}\frac{\alpha}{|x|^2} +\alpha\left(\alpha-1+\frac{2}{n-1}\right)\frac{|\nabla_M c|^2}{c^2}\right)
w^2
   \quad\mbox{on }M,
    \ee
     in a distributional sense. Moreover, 
     \be\label{e:better-a}\frac12\Delta_M w^2 - |\nabla_M w|^2 + |A|^2 w^2 \geq \alpha(\alpha+1)\frac{w^2}{|x|^2}\quad\mbox{on }M,\qquad \mbox{if }\alpha\geq 1-\frac{2}{n-1},
    \ee
     and if $f$ is strictly convex and $c$ is nontrivial, the previous inequalities become strict.
\end{lemma}
\begin{proof}
The proof is omitted as it involves considering the regularized sequence $w_\eps:=(c^2+\eps)^{\alpha/2}$, with $\eps>0$, and repeating the same computations as in the proof of Lemma \ref{l:alpha}.
\end{proof}
\subsection{The boundary inequality} Inspired by \cite[Section 4]{js}, we focus on the boundary behavior of a class of quadratic functions of the eigenvalues of $A$: given $\Lambda=(\lambda_1(A),\dots,\lambda_n(A))$, we consider 
\be\label{e:candidate4}
f(\Lambda):=\Bigg(\sum_{\lambda_i(A)\geq 0}\lambda_i(A)^2 + a \sum_{\lambda_s(A)<0}\lambda_s^2(A)\Bigg)^{1/2},\quad \mbox{with }a>0.
\ee
We set $c(x):=f(\Lambda(x))$. It is clear that, up to a regularized approximation as in Lemma \ref{l:alpha}, such a function fulfills all the assumptions of Lemma \ref{l:generalized-simon}.\\
For notational simplicity, throughout the section we omit the dependence of $c$ and $A$ on the spatial variable $x$. In particular, in this section we use the following notation for the eigenvalues of $A$:
\begin{enumerate}
    \item the subscript $i$, to denote non-negative eigenvalues;
    \item the subscript $s$, to denote negative eigenvalues;
    \item the subscript $j$, to denote an eigenvalue without sign assumption.
\end{enumerate}
\begin{lemma}\label{l:generalized-simon-boundary}
Let $E$ be a stationary point of $\mathcal{F}$ and $M:=\partial E$. Assume that $\partial M$ has an isolated singularity at the origin, and that $c$ is defined as in \eqref{e:candidate4}. Then, 
$$
\frac12(\nabla_M c^2\cdot \eta) = \cos(\theta) H_{\partial M}(c^2+ a|A|^2) + \cot(\theta)\left(\sum_i \lambda_i^3 + a \sum_s \lambda_s^3\right)\quad\mbox{on }(\partial M\setminus \{0\})\cap \{H_{\partial M}\geq 0\}
$$
and 
$$
\frac12(\nabla_M c^2\cdot \eta) = \cos(\theta) H_{\partial M}(c^2+ |A|^2) + \cot(\theta)\left(\sum_i \lambda_i^3 + a \sum_s \lambda_s^3\right)\quad\mbox{on }(\partial M\setminus \{0\})\cap \{H_{\partial M}\leq 0\}.
$$
Then, by choosing $\alpha \in (0,1)$ so that 
\be\label{e:L.alpha.a.zero}
\begin{aligned}
H_{\partial M} c^2(1-\alpha) - a\alpha H_{\partial M}|A|^2 -\frac{\alpha}{\sin(\theta)}\left(\sum_i \lambda_i^3 + a \sum_s \lambda_s^3\right) \geq 0&\quad\mbox{on }(\partial M\setminus \{0\})\cap \{H_{\partial M}\geq 0\},\\
H_{\partial M} c^2(1-\alpha) - \alpha H_{\partial M}|A|^2 -\frac{\alpha}{\sin(\theta)}\left(\sum_i \lambda_i^3 + a \sum_s \lambda_s^3\right) \geq 0&\quad 
\mbox{on }(\partial M\setminus \{0\})\cap \{H_{\partial M}\leq 0\},
\end{aligned}
\ee
we infer that $w:=c^\alpha$ satisfies the boundary inequality in \eqref{e:assumption-js} in a distributional sense.
We stress that in the set $(\partial M\setminus \{0\}) \cap \{H_{\partial M}\neq 0\}$, it is equivalent to choosing $\alpha \in (0,1)$ so that 
$$
H_{\partial M}(1-\alpha L)\geq 0 \quad \mbox{on }\partial M \setminus \{0\},
$$
where we set
\be\label{e:L.alpha.a}
L(x):= \begin{cases}
1 + \displaystyle \frac{a|A|^2}{c^2} + \frac{1}{\sin(\theta) H_{\partial M}c^2}\left(\sum_i \lambda_i^3(A) + a \sum_s\lambda_s^3(A)\right) &\mbox{if }H_{\partial M} > 0,\vspace{0.3cm}\\
1 + \displaystyle \frac{|A|^2}{c^2} + \frac{1}{\sin(\theta) H_{\partial M}c^2}\left(\sum_i \lambda_i^3(A) + a \sum_s\lambda_s^3(A)\right) &\mbox{if }H_{\partial M} < 0.
\end{cases}
\ee
\end{lemma}
\begin{proof}
Let $x_0 \in \partial M \setminus \{0\}$ and $\mathcal{N}_{x_0}$ be a neighborhood of $x_0$ in which we assume that $M$ is given by the graph of a non-negative $u$ as in Section \ref{s:coords}. For notational simplicity, we omit the dependence of the eigenvalues on $A$. Since in $\partial M\cap \mathcal{N}_{x_0}$ we can express the eigenvalues of $A$ in terms of $u$ as in \eqref{e:lambda}, we have
$$
\frac12\partial_n c^2 = \sum_{i}\lambda_i \partial_n a_{ii} + a\sum_{s}\lambda_s \partial_n a_{ss}.
$$
Then, at $x_0$ the following identities hold true: for every $j<n$ we have
\begin{align*}
\partial_n a_{jj} &= 
\partial_n\left(\frac{u_{jj}}{(1+u_n^2)^{1/2}}\right) = \frac{u_{jjn}}{(1+u_n^2)^{1/2}} - \frac{u_{jj}u_n u_{nn}}{(1+u_n^2)^{3/2}}\\
&= -\frac{1}{u_n}\frac{u_{jj}^2-u_{nn}u_{jj}}{(1+u_n^2)^{1/2}} - \frac{u_{jj}u_n u_{nn}}{(1+u_n^2)^{3/2}}\\
&= \cot(\theta)(1+u_n^2)^{1/2}\lambda_j^2 + H_{\partial M}\lambda_j
\end{align*}
and
\begin{align*}
\partial_n a_{nn}&=\partial_n\left(\frac{u_{nn}}{(1+u_n^2)^{3/2}}\right) = \frac{u_{nnn}}{(1+u_n^2)^{3/2}} - 3 \frac{u_n u_{nn}^2}{(1+u_n^2)^{5/2}}\\
&=\frac{1}{(1+u_n^2)^{3/2}}\left(-\cot(\theta)(1+u_{n}^2)\sum_{j<n} u_{jj}^2 - \cot(\theta)\frac{1+3u_{n}^2}{1+u_{n}^2}u_{nn}^2\right)
- 3 \frac{u_n u_{nn}^2}{(1+u_n^2)^{5/2}}\\
&= \frac{-\cot(\theta)}{(1+u_n^2)^{1/2}}\sum_{j<n}u_{jj}^2 - \cot(\theta)\frac{1+3u_n^2}{(1+u_n^2)^{5/2}}u_{nn}^2 + 3 \tan(\theta)\frac{u_{nn}^2}{(1+u_n^2)^{5/2}}\\
&= -\cot(\theta)(1+u_n^2)^{1/2}\sum_{j<n}\lambda_j^2 - \cot(\theta)\frac{u_{nn}^2}{(1+u_n^2)^{5/2}}\\
&= -\frac{1}{\sin(\theta)}\sum_{j<n}\lambda_j^2 - \frac{1}{\sin(\theta)}\lambda_n^2\\
&= \, \frac{H_{\partial M}}{\lambda_n}\sum_{j=1}^n \lambda_j^2,
\end{align*}
where in the last line we use the last equation in Lemma \ref{l:H}. Then, at $x_0$, we have
\begin{align*}
\frac12 \partial_n c^2 =&\, H_{\partial M} c^2 -  \left( \frac{1+a}{2} + \frac{1-a}{2}\frac{\lambda_n}{|\lambda_n|}\right)H_{\partial M}\lambda_n^2\\
& + \cot(\theta)(1+u_n^2)^{1/2}\left(\sum_{i\neq n} \lambda_i^3 + a\sum_{s\neq n} \lambda_s^3\right) + \left( \frac{1+a}{2} + \frac{1-a}{2}\frac{\lambda_n}{|\lambda_n|}\right) H_{\partial M}\sum_j \lambda_j^2\\
=&\, H_{\partial M} c^2 +\frac{1}{\sin(\theta)}\left(\sum_{i\neq n} \lambda_i^3 + a\sum_{s\neq n} \lambda_s^3 + 
\left( \frac{1+a}{2} + \frac{1-a}{2}\frac{\lambda_n}{|\lambda_n|}\right) \lambda_n^3
\right)\\
&+ \left( \frac{1+a}{2} + \frac{1-a}{2}\frac{\lambda_n}{|\lambda_n|}\right) H_{\partial M}\sum_j \lambda_j^2,
\end{align*}
where in the last equality we used that $\lambda_n =-\sin(\theta)H_{\partial M}$ at $x_0$. Thus, if $\lambda_n(x_0)\neq 0$, we have two possibilities according to the sign of $\lambda_n$:
\begin{enumerate}
    \item[(\rm i)] if $\lambda_n(x_0)<0$
\begin{align*}
\frac12 (\nabla_M c^2\cdot \eta) &= \cos(\theta) H_{\partial M}(c^2+ a|A|^2) + \cot(\theta)\left(\sum_i \lambda_i^3 + a \sum_s \lambda_s^3\right)\\
&= \cos(\theta)H_{\partial M}c^2\left[1 + \frac{a\lambda_n \sum_j \lambda_j^2-\sum_i \lambda_i^3 - a \sum_s \lambda_s^3}{\lambda_n \sum_i \lambda_i^2 + a \lambda_n \sum_s \lambda_s^2}\right]\quad\mbox{at }x_0;
\end{align*}
    \item[(\rm ii)] if $\lambda_n(x_0)>0$
    \begin{align*}
\frac12 (\nabla_M c^2\cdot \eta) &= \cos(\theta) H_{\partial M}(c^2+ |A|^2) + \cot(\theta)\left(\sum_i \lambda_i^3 + a \sum_s \lambda_s^3\right)\\
 &=\cos(\theta)H_{\partial M}c^2\left[1+\frac{\lambda_n \sum_{j} \lambda_j^2 - \sum_i \lambda_i^3 - a \sum_s \lambda_s^3}{\lambda_n \sum_i \lambda_i^2 + a \lambda_n \sum_s \lambda_s^2}\right]
\quad\mbox{at }x_0.
\end{align*}
\end{enumerate}
By exploiting again the relation between $\lambda_n$ and $H_{\partial M}$, we get the desired conditions in terms of $H_{\partial M}$. Finally, the boundary inequality in \eqref{e:assumption-js}  follows by considering the regularized sequence $w_\eps:=(c^2+\eps)^{\alpha/2}$, with $\eps>0$, and by repeating the same computation in the proof of Lemma \ref{l:boundary-L}.
\end{proof}
\begin{remark}[The Jerison-Savin approach]\label{remark2}
It is natural to observe a further connection between capillary surfaces with free boundary and the one-phase Bernoulli free boundary problem. By comparing the different stability inequalities, it is immediate to notice that the eigenvalues of the second fundamental form $A$ play the same role as the eigenvalues of the Hessian matrix of a minimizer for the one-phase problem. Indeed, if $u$ is such a minimizer, since the mean curvature of the free boundary is non-negative (see \cite[Remark 2]{cjk}), then given $f$ as in \eqref{e:candidate4}, the function
    $$
    c(x):=f(\Lambda(x)),\qquad\mbox{with}\quad\Lambda(x)=\Big(\lambda_1(\nabla^2 u(x)),\dots,\lambda_n(\nabla^2 u(x))\Big)
    $$
    satisfies the following differential inequalities (see \cite[Section 4.2 and Section 4.3]{js})
$$
\frac12\Delta c^2 -|\nabla c|^2 \geq \frac{2}{n-1}|\nabla c|^2 + 2\frac{n-2}{n-1}\frac{c^2}{|x|^2} \qquad\mbox{in }\{u>0\}
$$
and
    $$
\frac12(\nabla c^2\cdot \bar{\nu}) = c^2 H_{\partial\{u>0\}}\left[1+\frac{a\lambda_n \sum_{j} \lambda_j^2 - \sum_i \lambda_i^3 - a \sum_s \lambda_s^3}{\lambda_n \sum_i \lambda_i^2 + a \lambda_n \sum_s \lambda_s^2}\right], \quad \mbox{on }\partial \{u>0\}\setminus \{0\},
$$
where $\bar{\nu}$ is the unit normal vector field of $\partial \{u>0\}$ that points out of $\{u>0\}$. Note that the fact that the mean curvature of the free boundary $\partial\{u>0\}$ has a constant sign is consistent with the analogous result for graphical solutions in Theorem \ref{thm-main-graphical-cone}. Thus, such similarities suggest that the study of singularities in the two free boundary problems is equivalent, independently of the value of the contact angle.\\
Ultimately, the non-existence of stable singularities is deduced by exploiting the interplay between the interior and the boundary inequalities in \eqref{e:assumption-js}, which are optimally achieved by different configurations (respectively, the axially symmetric case and the Lawson-type cone in \cite[Example 9]{hong}). 
\end{remark}

\subsection{Proof of Theorem \ref{thm-main}}	
In general, it remains to consider the cases $n\geq 4$, since the lower-dimensional cases have already been covered by Section \ref{s:cel} and \cite{taylor}.
For notational simplicity, throughout the proof we often omit the dependence on the spatial variable $x$ when it is evident. Now,  given \(a=(n-2)^2\), set
\be\label{e:definito1}
c:=\left(\sum_{i}\lambda_i^2+(n-2)^2\sum_{s}\lambda_s^2\right)^{1/2},\qquad w:=c^\alpha,
\ee
where the exponent $\alpha\in (0,1)$ will be chosen separately in the two cases below.

First, by Lemma \ref{l:generalized-simon-2} (see \eqref{e:better-a}), we have
$$
\frac12\Delta_M w^2 - |\nabla_M w|^2 + |A|^2 w^2 \geq \alpha(\alpha+1)\frac{w^2}{|x|^2}\quad\mbox{on }M,\qquad \mbox{if }\alpha\geq 1-\frac{2}{n-1},  
$$
where the previous differential inequality is satisfied in a distributional sense. 
On the other hand, as regards the boundary inequality, before applying Lemma \ref{l:generalized-simon-boundary}, it is better to address the boundary behavior of the competitor at a point $x_0 \in (\partial M \setminus \{0\})\cap \{H_{\partial M}=0\}$. Thus, by substituting the definition of $w$ in \eqref{e:L.alpha.a.zero}, we want to show that 
\be\label{e:serve}
-\left(\sum_i \lambda_i^3 + (n-2)^2\sum_s \lambda_s^3\right) \geq 0\quad\mbox{on }(\partial M\setminus \{0\})\cap \{H_{\partial M}= 0\}.
\ee
Since $M$ is a smooth cone outside the origin, it is not restrictive to choose the coordinates at $x_0$ in such a way that $x_0=x_0^1e_1$ where $e_1$ is the radial direction. Thus, at $x_0$ we have $\lambda_1 = 0, \lambda_n=0$ (since $\lambda_n =-\sin(\theta) H_{\partial M}$) and, without loss of generality, we assume that $\lambda_2 \geq 0$. Then, since $\sum_{j=2}^{n-1}\lambda_j =0 $ at $x_0$, we get $\sum_i \lambda_i = \sum_s |\lambda_s|$ at $x_0$ and
$$
\sum_i \lambda_i^3 \leq \left(\sum_i \lambda_i\right)^3 =\left(\sum_s |\lambda_s|\right)^3 \leq (n-2)^2\sum_s |\lambda_s|^3 \quad \text{at }x_0,
$$
where the first inequality follows from monotonicity and the second from H\"{o}lder's inequality.

Therefore, we can restrict to the case $H_{\partial M}\neq 0$ and we proceed by computing the function $L$ (see \eqref{e:L.alpha.a} for its precise definition) on $\partial M\cap \mathbb{S}^{n-1}\cap \{H_{\partial M}\neq 0\}$. 

Thus, on $\partial M\cap \mathbb{S}^{n-1}\cap \{H_{\partial M}\neq 0\}$ we have
$$
L(x):= \begin{cases}
1 + \displaystyle{\frac{(n-2)^2 \lambda_n \sum_j \lambda_j^2-\sum_i \lambda_i^3 - (n-2)^2 \sum_s \lambda_s^3}{\lambda_n \sum_i \lambda_i^2 + (n-2)^2 \lambda_n \sum_s \lambda_s^2}} &\mbox{if }H_{\partial M} > 0 \,\, (\mbox{i.e. }\lambda_n(A)< 0), \vspace{0.3cm}\\
1 + \displaystyle{\frac{\lambda_n \sum_j \lambda_j^2-\sum_i \lambda_i^3 - (n-2)^2\sum_s \lambda_s^3}{\lambda_n \sum_i \lambda_i^2 + (n-2)^2\lambda_n \sum_s \lambda_s^2}} &\mbox{if }H_{\partial M} < 0 \,\, (\mbox{i.e. }\lambda_n(A)> 0).
\end{cases}
$$
Moreover, if for every $j=1,\dots,n-1$ we consider the normalized eigenvalue $\xi_j:=-\lambda_j/\lambda_n$, we get
$$
L(x) = \begin{cases}
1 + \displaystyle\frac{(n-2)^2\sum_j \xi_j^2 + \sum_i \xi_i^3 + (n-2)^2 \sum_s \xi_s^3}{(n-2)^2+\sum_i \xi_i^2 + (n-2)^2\sum_s \xi_s^2}&\mbox{if }\lambda_n(A)< 0, \vspace{0.3cm}\\
1 + \displaystyle\frac{\sum_j \xi_j^2 + (n-2)^2\sum_i \xi_i^3 +\sum_s \xi_s^3 }{1+ (n-2)^2\sum_i \xi_i^2 +\sum_s \xi_s^2 } &\mbox{if }\lambda_n(A)> 0.
\end{cases}
$$
We recall that
\begin{enumerate}
    \item the subscript $i$, denotes non-negative normalized eigenvalues;
    \item the subscript $s$, to denote negative normalized  eigenvalues;
    \item the subscript $j$, to denote an normalized eigenvalue without sign assumption.
\end{enumerate}
Secondly, since $\mathrm{tr}(A)=0$ we infer that $\sum_j \xi_j =1$ and, since $M$ is a smooth cone outside the origin, it is not restrictive to choose the coordinates at $x_0\in \partial M\cap \mathbb{S}^{n-1}$ so that $x_0=x_0^1e_1$, where $e_1$ is the radial direction. Thus, $\lambda_1(A)\equiv0$ on $\partial M\cap \mathbb{S}^{n-1}$, which implies that $\xi_1=0$. 

To conclude the proof, we need to address separately two cases.\\

Case 1: $H_{\partial M}\leq 0$ on $\partial M\setminus \{0\}$ and $n\leq 6$. We first show that   
\be\label{e:minL}
\inf_{\substack{x \in \partial M \cap \mathbb{S}^{n-1} \\ H_{\partial M}(x)< 0}} L(x)
\geq \min
\left\{1 + \displaystyle\frac{\sum_j \xi_j^2 + (n-2)^2\sum_i \xi_i^3 -\sum_s |\xi_s|^3 }{1+ (n-2)^2\sum_i \xi_i^2 +\sum_s \xi_s^2 } \colon \sum_j \xi_j =1, \xi_1=0\right\}=1+\frac{1}{n-2}.
\ee
For notational convenience, set $a=(n-2)^2$. Then, it is sufficient to show that $G(\xi)\geq 0$ for every $\xi \in \R^{n-1}$ satisfying $\sum_j \xi_j=1,\xi_1=0$, where 
$$
G(\xi):=(n-2)\sum_i \xi_i^2
+(n-3)\sum_s |\xi_s|^2
 -(n-2)\sum_s |\xi_s|^3  -1  
+ a\left((n-2)\sum_i \xi_i^3-\sum_i \xi_i^2\right).
$$
First, we proceed by showing that the last term is non-negative. Indeed, by the Cauchy-Schwarz inequality, we have
\be\label{e:cs-xi}
\sum_i \xi_i^3 \geq 
\frac{\left(\sum_i \xi_i^2\right)^{2}}{\sum_i \xi_i},\qquad \sum_i \xi_i^2 \geq \frac{1}{\#_i}\left(\sum_i \xi_i\right)^2,
\ee
where $\#_i$ is the number of positive coordinates of $\xi$. Since $\#_i\leq n-2$ and $\sum_i \xi_i = 1+\sum_s |\xi_s| \geq 1$, it implies that 
$$
(n-2)\sum_i \xi_i^3 -\sum_i \xi_i^2 \geq \sum_i \xi_i^2 \left((n-2)\frac{\sum_i \xi_i^2}{\sum_i \xi_i} - 1\right) \geq 0,
$$
which shows that $G(\xi)$ is non-decreasing in the variable $a$.

Now, we notice that if $\sum_s|\xi_s| = 0$ the inequality is trivially satisfied. Indeed, in this case $\sum_i \xi_i = 1$ and 
$$
\sum_i \xi_i^2 \geq \frac{1}{\#_i}\left(\sum_i \xi_i\right)^2 = \frac{1}{n-2}.
$$
By substituting the previous estimate, we infer 
$$
G(\xi) \geq a\left((n-2)\sum_i \xi_i^3-\sum_i \xi_i^2\right)\geq 0,
$$
as we proved above. Assume now that $\sum_s |\xi_s|\neq 0$, then $\#_i\leq n-3$ and, by refining \eqref{e:cs-xi}, we get
$$
\sum_i \xi_i^2 \geq \frac{1}{n-3}\left(\sum_i \xi_i\right)^2 
$$
which yields that 
$$
(n-2)\sum_i \xi_i^3 -\sum_i \xi_i^2 \geq \frac{1}{n-3}\left(\sum_i \xi_i\right)^2\left(\frac{n-2}{n-3}\sum_i \xi_i - 1\right).
$$
Finally, by combining the previous estimates with  
$$
(n-3)\sum_s |\xi_s|^2  -(n-2)\sum_s |\xi_s|^3 \geq  -(n-2)\sum_s |\xi_s|^3  \geq -(n-2)\left(\sum_s |\xi_s|\right)^3 
$$
we infer that
\begin{align*}
G(\xi) &\geq \frac{n-2}{n-3}\left(\sum_i \xi_i\right)^2-1-(n-2)\left(\sum_s |\xi_s|\right)^3 + \frac{a}{(n-3)}\left(\sum_i \xi_i\right)^2\left(\frac{n-2}{n-3}\sum_i \xi_i - 1\right)\\
%&\geq \frac{n-2}{n-3}\left(\sum_i \xi_i\right)^2 -1-(n-2)\left(\sum_s |\xi_s|\right)^3 + (n-3)\left(\sum_i \xi_i\right)^2\left(\frac{n-2}{n-3}\sum_i \xi_i - 1\right),
&\geq \frac{n-2}{n-3} + \frac{(n-2)^2+3(n-2)-2}{n-3}\sum_s|\xi_s| + \frac{2(n-2)^2-1}{n-3}\left(\sum_s |\xi_s|\right)^2 
\end{align*}
where in the second inequality we use that $a>(n-3)^2$ and that $\sum_i \xi_i = 1+ \sum_s |\xi_s|$. Since $\sum_s|\xi_s|>0$ and $n\geq 4$, all the coefficients on the right-hand side are positive. Hence $G(\xi)\geq 0$, as we claimed. Notice that equality is attained at $\xi_1=0,\xi_2=\cdots=\xi_{n-1}=1/(n-2)$, which implies sharpness of the lower bound above.

Therefore, if we choose $\alpha=\frac{n-2}{n-1}$ (which satisfies the lower bound in \eqref{e:better-a}), we obtain 
$$
H_{\partial M}\left(1-\alpha L\right)\geq H_{\partial M}\left(1-\alpha \frac{n-1}{n-2}\right)\geq0
\quad\mbox{on }\partial M\cap \mathbb{S}^{n-1}\cap \{H_{\partial M}<0\}.
$$
Together with \eqref{e:serve}, this proves the boundary inequality
also at the points where $H_{\partial M}=0$. Finally, as in the proof of Theorem \ref{thm-axially}, the condition \eqref{e:check} can be rephrased as 
\be\label{e:Lambda-hard}
\Lambda = \frac{(n-2)(2n-3)}{(n-1)^2} \geq  \left(\frac{n}{2}-\frac{n-2}{n-1}-1\right)^2 \quad\longleftrightarrow\quad 2 \leq n \leq 6.
\ee
Since the first inequality is strict for $2<n<6$, if $n=4,5$ we conclude by directly applying Proposition \ref{p:JS}, which implies that $w\equiv 0$ and that $M$ is flat.

If $n=6$, we have $\alpha =4/5, \Lambda =36/25$ and the first inequality in \eqref{e:Lambda-hard} holds with equality. In order to apply Proposition \ref{p:JS}, we show that boundary inequality is strict at some point of the free boundary, leading to the flatness of $M$. The argument is similar to the one used in the last part of of the proof of Theorem \ref{thm-axially}.

By the previous analysis, we know that the minimum in \eqref{e:minL} is attained by
$$
\lambda_1 = 0,\quad \lambda_2=\dots=\lambda_{5}=-\frac{\lambda_6}{4}<0,\quad \lambda_6>0.
$$
At such configurations, we have
$$
c^2 = 5\lambda_6^2,\quad |A|^2 = \frac54 \lambda_6^2,
\quad \sum_i\lambda_i^3 + 16 \sum_s \lambda_s^3 =0.
$$
Hence, by Lemma \ref{l:generalized-simon-boundary},
$$
\frac12(\nabla_M c^2\cdot \eta) = \cos(\theta) H_{\partial M}(c^2+ |A|^2) = -\frac{25}{4}\cot(\theta)\lambda_6^3\neq 0,\quad \text{for }\theta \in \left(0,\frac{\pi}{2}\right).
$$
Therefore, $\nabla_M c \cdot \eta \neq 0$
in a neighborhood of each point of $\partial M\setminus\{0\}$ at which the minimum in \eqref{e:minL} is attained. 
On the other hand, Lemma \ref{l:generalized-simon-2} (see \eqref{e:alpha}), gives
\begin{align*}
\frac12\Delta_M w^2 - |\nabla_M w|^2 +|A|^2 w^2 
&\geq \frac{32}{25}\frac{w^2}{|x|^2}  + \frac{4}{25}\frac{|\nabla_M c|^2}{c^2} w^2\\
&\geq \frac{36}{25}\frac{w^2}{|x|^2}  + \frac{4}{25}\frac{(\nabla_M c\cdot \xi)^2}{c^2} w^2\qquad\mbox{on }M,
\end{align*}
for every $\xi \in \mathbb{S}^{n-1}$ satisfying $\xi \cdot x=0$.
Taking $\xi$ as an extension of $\eta$ in $M$, we deduce that 
interior Simons-type inequality is strict in a neighborhood of every point of every point of $\partial M\setminus \{0\}$ at which the minimum in \eqref{e:minL} is attained. At all other boundary points, the boundary inequality is strict. Consequently, the strictness condition in Proposition \ref{p:JS} is satisfied, and we conclude that $w\equiv0$, hence $M$ is flat.

\medskip

Case 2: $H_{\partial M}\geq 0$ on $\partial M\setminus \{0\}$ and $n=4$. Take
$$
\alpha=\frac13,\quad w=c^{1/3},\quad c=\left(\sum_{i}\lambda_i^2+4\sum_{s}\lambda_s^2\right)^{1/2}.
$$
Direct computation implies that 
\begin{align*}
\sup_{\substack{x \in \partial M \cap \mathbb{S}^3\\ H_{\partial M}(x)> 0} }L(x) &\leq \max
\left\{1 + \frac{4 \sum_{j} \xi_j^2 +\sum_i \xi_i^3 +4 \sum_{s } \xi_s^3}{\sum_i \xi_i^2 + 4 \sum_s \xi_s^2 + 4}\colon \sum_j \xi_j =1, \xi_1=0\right\}= \max\{K_1, K_2 \} 
\end{align*}
where 
\begin{align}
K_1 :=& \max_{t\in [0,1]}\left(1+\frac{4 (t^2 + (1-t)^2) + t^3 + (1-t)^3}{t^2+(1-t)^2 + 4}\right)=2,\notag\\K_2 :=& \max_{t>1}\left(1+  \frac{4(t^2+(1-t)^2)+t^3+4(1-t)^3}{t^2+4(1-t)^2+4}\right)=3.\label{e:K2}
\end{align}
In this case we need to address separately the points on the boundary where $K_2$ is attained. Indeed, since $L\leq 3$ on $\partial M\cap \mathbb{S}^3\cap \{H_{\partial M}>0\}$, the choice $\alpha:=1/3$ is admissible and the differential inequalities in \eqref{e:assumption-js} hold in a distributional sense. Notice that, in this case, the condition \eqref{e:check} is an equality, indeed  
\be\label{e:reph}
\frac49 = \Lambda = \left(\frac{n}{2}-\alpha-1\right)^2 =  \frac49.
\ee
Therefore, in order to apply Proposition \ref{p:JS}, it remains  to show that in a neighborhood of those points at which the boundary condition is an equality, the interior condition can be strictly improved. We emphasize that condition \eqref{e:check} cannot be improved as in the case of an axially symmetric cone since
$$
\frac13 = \alpha = 1 - \frac{2}{n-1} = \frac13,
$$
which is an alternative lower bound for the admissible exponent $\alpha$ (see \eqref{e:better-a}).

It is immediate to observe that the boundary condition is not strict at points $x_0 \in \partial M \cap \mathbb{S}^3$ where $K_2$ is attained (i.e. $t=2$ in \eqref{e:K2}), which corresponds, up to relabeling the eigenvalues of $A$, to the case 
$$
\lambda_1 =0,
\quad \lambda_2=-2\lambda_4,
\quad \lambda_3=\lambda_4.
$$ 
Hence, since at $x_0$ we have $\lambda_2\neq \lambda_3$ and $\lambda_4\neq 0$, we infer that in a neighborhood $\mathcal{N}_{x_0}$ of $x_0$, we have that the eigenvalues $\lambda_j$, with $j\neq 4$, are not all equal. Therefore, it is possible to improve \eqref{e:kneq1} for the case $k=4$: indeed, there exists $\delta>0$, depending on $\mathcal{N}_{x_0}$, such that  
$$
c_4^2 \leq \left(\frac32 + \delta\right)\left(\sum_{i\neq \{1,4
\}}\partial^2_{a_{i4}} F(A) u_{ii4}^2\right) \frac{c}{2} < \frac{3}{2}\left(\sum_{i\neq \{1,4
\}}\partial^2_{a_{i4}} F(A) u_{ii4}^2\right) c\quad \mbox{in }\mathcal{N}_{x_0}.
$$
Thus, by continuing the proof as in Lemma \ref{l:generalized-simon}, we infer that the interior inequality holds strictly.
\begin{remark}[The new competitor]\label{rmk:competitors}
To rule out the presence of singularities in dimension 
$n=3$, it suffices to use powers of the squared norm of the second fundamental form $|A|$ as a competitor. On the other hand, we emphasize that the competitors of the form \eqref{e:candidate4} are necessary for the study of the higher dimensional case (note that $|A|$ corresponds to the case $a=1$ in \eqref{e:candidate4}).\\
Retracing the proof of Theorem \ref{thm-main}, we see that the choice in \eqref{e:definito1} is the only one among those in \eqref{e:candidate4} that enables the application of the stability criterion \emph{à la} Jerison-Savin.
\end{remark}
\subsection{Proof of Theorem \ref{thm-main-graphical-cone}}   
In light of Theorem \ref{thm-main}, it is sufficient to show that free boundaries associated to graphical cones have mean curvature with constant sign. Precisely, if we assume that $M$ is a graph in the $e_{n+1}$ direction, then the function $w:= \nu \cdot e_{n+1}$ must be positive.

On the other hand, since $w$ is a Jacobi field, precisely $$\Delta_M w + |A|^2 w=0\qquad \text{on $M$},$$ its restriction to $\partial B_1$ must attain a minimum on $\partial M \cap \partial B_1$. Indeed, since $w$ is $0$-homogeneous, by the maximum principle the minimum cannot be attained in the interior of $M\cap \partial B_1$. 

Finally, since $w$ achieves its minimum on $\partial M\cap \partial B_1$, the outer normal derivative must be non-positive, that is
$$
0 \geq (\nabla_M w\cdot \eta) = -A(\eta,\eta)(\eta\cdot e_{n+1}) = \sin(\theta) A(\eta,\eta) = -\sin^2(\theta) H_{\partial M} \quad \mbox{on }\partial M \cap \partial B_1.
$$
Therefore, $H_{\partial M}\geq 0$ and we can apply Theorem \ref{thm-main}.
\appendix

\section{Symmetric functions of eigenvalues}\label{a:eigen}
In this section we recall some useful results for symmetric functions of eigenvalues of a given matrix. Finally, these results will then be applied in Section \ref{s:finally}.
\begin{definition}\label{d:symmetric}
We say that a function $f:\R^n\to\R$ is symmetric if for any permutation $\pi:\{1,\dots,n\}\to\{1,\dots,n\},$
and any $(x_1,\dots,x_n)\in\R^n$,
we have that
\begin{equation}\label{e:definition-symmetric-function}
f\big(x_{\pi(1)},x_{\pi(2)},\dots,x_{\pi(n)}\big)=f(x_1,x_2,\dots,x_n)\,.
\end{equation}
\end{definition}
We notice that if $x_0=(x_0^1,\dots,x_0^n)\in\R^n$ is such that $x_0^i=x_0^j$ for some $i\neq j$, then the symmetry of $f$ implies
$$f(x_0+te_i)\equiv f(x_0+te_j),\quad\text{for every}\quad t\in\R.$$
In particular, if $f\in C^2$, this implies that  $\partial_if(x_0)=\partial_j f(x_0)$ and $\partial_{ii}f(x_0)=\partial_{jj}f(x_0).$
Similarly,
\begin{align*}
f(x_0+Te_i+Se_j)
&=f(x_0)+(S+T)\partial_if(x_0)+\frac{\mathbb{S}^2+T^2}2\partial_{ii}f(x_0)+ST\partial_{ij}f(x_0)+o(\mathbb{S}^2+T^2)
\end{align*}
Given a symmetric matrix $A\in S(n)$, we denote by $\lambda_1(A),\dots,\lambda_n(A)$ the $n$ real (possibly equal) eigenvalues of $A$. In particular, by identifying a matrix $A \in S(n)$ with a point $A=(a_{ij})_{1\le i\le j\le n}$ of $\R^{(n+1)n/2}$, it is possible to define a symmetric function of eigenvalues $F:\R^{n(n+1)/2}\to\R,$
as
\begin{equation}\label{e:definition-of-F}
F(A):=f\Big(\lambda_1(A),\dots,\lambda_n(A)\Big),
\end{equation}
with $f\colon \R^n \to \R$ being symmetric. We start with the following lemma.
\begin{lemma}\label{l:main-lemma-derivatives}
Let $f:\R^n\to\R$ be a symmetric $C^2$ function and  $F:\R^{n(n+1)/2}\to\R$ be defined by \eqref{e:definition-of-F}. Given a diagonal matrix $A_0\in S(n)$, we denote by $\Lambda(A_0)$ the collection of its eigenvalues
$$\Lambda(A_0)=\big(\lambda_1(A_0),\dots,\lambda_n(A_0)\big).$$
Then, $F$ is twice differentiable at $A_0$ and the following identities hold true:
\begin{align*}
\partial_{a_{ij}}F(A_0)&=\begin{cases}
\partial_if\big(\Lambda(A_0)\big) & \quad\text{if}\quad i=j\,;\\
0 & \quad\text{if}\quad i< j\,.
\end{cases}
\end{align*}
\begin{align*}
\partial_{a_{ij}}^2F(A_0)
=\begin{cases}
\partial_{ii}f\big(\Lambda(A_0)\big)&\quad\text{if}\quad i=j\,; \medskip\\
\displaystyle 2\frac{\partial_jf\big(\Lambda(A_0)\big)-\partial_if\big(\Lambda(A_0)\big)}{\lambda_j(A_0)-\lambda_i(A_0)}&\quad\text{if}\quad i< j\quad\text{and}\quad \lambda_i(A_0)\neq\lambda_j(A_0)\,;\medskip\\
2\partial_{ii}f\big(\Lambda(A_0)\big)
-2\partial_{ij}f\big(\Lambda(A_0)\big)&\quad\text{if}\quad i< j\quad\text{and}\quad \lambda_i(A_0)=\lambda_j(A_0)\,.
\end{cases}
\end{align*}
\begin{align*}
\partial^2_{a_{ij}a_{i'j'}}F(A_0)
=\begin{cases}
\partial_{ii'}f(\Lambda(A_0))&\quad\text{if}\quad i=j\neq i'=j'\ ;\\
0&\quad\text{in all other cases in which}\quad \{i,j\}\neq\{i',j'\}.
\end{cases}
\end{align*}
\end{lemma}	
\begin{proof}
We first compute the partial derivatives $\partial_{a_{ij}}F(A_0)$ and $\partial^2_{a_{ij}}F(A_0)$.
Since the case $i=j$ is immediate to check, we consider $i<j$ and, without loss of generality, we can suppose that $i=1, j=2$ and $n=2$. Thus, for any $t\in\R$, the eigenvalues of the matrix
$$\begin{pmatrix}\lambda_1&t\\
t&\lambda_2\end{pmatrix}$$
are the solutions $\lambda_{1}(t)$ and $\lambda_2(t)$ to $(x-\lambda_{1})(x-\lambda_{2})-t^2=0,$ that is
\begin{align*}
\lambda_{1}(t) :=\frac{(\lambda_1+\lambda_2)+\sqrt{(\lambda_1-\lambda_2)^2+4t^2}}{2}\qquad\text{and}\qquad
\lambda_{2}(t) :=\frac{(\lambda_1+\lambda_2)-\sqrt{(\lambda_1-\lambda_2)^2+4t^2}}{2}
\end{align*}
We consider two cases:\\

\noindent  {\bf Case 1.} $\lambda_2\neq \lambda_1$ (without loss of generality, let us assume $\lambda_2>\lambda_1$). Then
\begin{align*}
\lambda_{1}(t)
=\lambda_2+\frac{t^2}{\lambda_2-\lambda_1}+o(t^2),\qquad\mbox{and}\qquad
\lambda_2(t)
=\lambda_1+\frac{t^2}{\lambda_1-\lambda_2}+o(t^2).
\end{align*}
Thus, we have the following Taylor expansion
$$f\big(\lambda_{1}(t),\lambda_{2}(t)\big)=f\big(\lambda_1,\lambda_2\big)+t^2\frac{\partial_2f\big(\lambda_1,\lambda_2\big)-\partial_1f\big(\lambda_1,\lambda_2\big)}{\lambda_2-\lambda_1}+o(t^2),$$
which gives
$$\partial_{a_{ij}}F(A_0)=0\qquad\text{and}\qquad \partial_{a_{ij}}^2F(A_0)=2\frac{\partial_jf(\Lambda(A_0))-\partial_if(\Lambda(A_0))}{\lambda_j(A_0)-\lambda_{i}(A_0)}\,.$$
\noindent {\bf Case 2.} $\lambda_1=\lambda_2$. Then $
\lambda_{2}(t) =\lambda_2+t$ and $\lambda_{1}(t)=\lambda_1-t$,
which implies
$$f(\lambda_{1}(t),\lambda_{2}(t))=f(\lambda_{1},\lambda_{2})+\frac{t^2}{2}\Big(f_{11}(\lambda_{1},\lambda_{2})+f_{22}(\lambda_{1},\lambda_{2})-2f_{12}(\lambda_{1},\lambda_{2})\Big)+o(t^2).$$
By direct computations, we get
$$\partial_{a_{ij}}F(A_0)=0\qquad\text{and}\qquad \partial_{a_{ij}}^2F(A_0)=\partial_{ii}f(\Lambda(A_0))+\partial_{jj}f(\Lambda(A_0))-2\partial_{ij}f(\Lambda(A_0))\,.$$
\phantom{a}\\
Now, it remains to compute the partial derivatives $\partial^2_{a_{ij} a_{i'j'}} F (A_0)$,
in the case $\{i,j\}\neq \{i',j'\}$. If $i,j,i'$ and $j'$ are all different, then we can compute the whole spectrum as in the previous case. Conversely, if two of the indices coincide, we can restrict our attention to two cases:\\

\noindent\bf Case Mixed.1. \rm $i=i'=j=1$ and $j'=2$.
For every $(t,s)\in\R^2$, we consider the matrix
$$
\begin{pmatrix}
\lambda_1+t& s\\
s&\lambda_2\\
\end{pmatrix}.$$
Its eigenvalues are the solutions to
$$\big(x-(\lambda_1+t)\big)\big(x-\lambda_2\big)-s^2=0,\qquad\big(x-\lambda_1\big)\big(x-\lambda_2\big)-t\big(x-\lambda_2\big)-s^2=0.$$
We consider two subcases.\vspace{0.1cm}\\
\noindent\bf Case Mixed.1.1. \rm $\lambda_1=\lambda_2$. Then the solutions are
$$\lambda_1(s,t)=\lambda_1+\frac{t+\sqrt{t^2+4s^2}}2\qquad\text{and}\qquad\lambda_2(s,t)=\lambda_2+\frac{t-\sqrt{t^2+4s^2}}2.$$
In this case
$$f\big(\lambda_{1}(s,t),\lambda_{2}(s,t)\big)=f(\lambda_1,\lambda_2)+t\partial_1f(\lambda_1,\lambda_2)+\left(\frac{t^2}{2}+s^2\right)\partial_{11}f(\lambda_1,\lambda_2)-s^2\partial_{12}f(\lambda_1,\lambda_2)+o(s^2+t^2),$$
which implies $\partial_{a_{11}a_{12}}F(A_0)=0.$\vspace{0.1cm}\\
\noindent\bf Case Mixed.1.2. \rm $\lambda_1\neq\lambda_2$. Then, one can check that
$$\lambda_2(s,t)=\lambda_2+\frac{s^2}{\lambda_2-\lambda_1}+o(s^2+t^2),\qquad\mbox{and}\qquad
\lambda_1(s,t)=\lambda_1+t+\frac{s^2}{\lambda_1-\lambda_2}+o(s^2+t^2).$$
Then,
\begin{align*}
	f\big(\lambda_{1}(s,t),\lambda_{2}(s,t)\big)&=f(\lambda_1,\lambda_2)+\bigg(t+\frac{s^2}{\lambda_1-\lambda_2}\bigg)\partial_1f(\lambda_1,\lambda_2)\\
	&\qquad +\frac{s^2}{\lambda_1-\lambda_2}\partial_2f(\lambda_1,\lambda_2)+\frac12t^2\partial_{11}f(\lambda_1,\lambda_2)+o(s^2+t^2),
	\end{align*}
so also in this case $\partial_{a_{11}a_{12}}F(A_0)=0$.\\

\noindent\bf Case Mixed.2. \rm $i=i'=1$,  $j=2$, $j'=3$. For every $(t,s)\in\R^2$, we consider the matrix
$$\begin{pmatrix}
\lambda_1& t&s\\
t&\lambda_2&0\\
s&0&\lambda_3
\end{pmatrix}$$
whose characteristic polynomial is defined as
\begin{align*}
P(x)
:=(x-\lambda_1)(x-\lambda_2)(x-\lambda_3)-t^2(x-\lambda_3)-s^2(x-\lambda_2).
\end{align*}
\bf Case Mixed.2.1. \rm $\lambda_2=\lambda_3\neq\lambda_1$. Then
$P(x)
=(x-\lambda_2)\big(x^2-(\lambda_1+\lambda_2)x+\lambda_1\lambda_2-s^2-t^2\big)$
and the eigenvalues are given by:
$$\lambda_1(s,t)=\lambda_1+\frac{t^2+s^2}{\lambda_1-\lambda_2}+o(t^2+s^2)\ ;\quad\lambda_2(s,t)=\lambda_2+\frac{t^2+s^2}{\lambda_2-\lambda_1}+o(t^2+s^2)\ ;\quad\lambda_3(s,t)\equiv\lambda_2.$$
Then,
$$f\big(\lambda_{1}(s,t),\lambda_{2}(s,t)\big)=f(\lambda_1,\lambda_2)+\frac{t^2+s^2}{\lambda_1-\lambda_2}\partial_1f(\lambda_1,\lambda_2)+\frac{t^2+s^2}{\lambda_2-\lambda_1}\partial_2f(\lambda_1,\lambda_2)+o(s^2+t^2),$$
and so we get $\partial_{a_{12}a_{13}}F(A_0)=0.$\vspace{0.1cm}\\
\noindent\bf Case Mixed.2.2. \rm $\lambda_2=\lambda_3=\lambda_1=\lambda$. Then
$$\lambda_{1}(s,t)=\lambda+\sqrt{s^2+t^2}\ ,\qquad \lambda_{2}(s,t)=\lambda-\sqrt{s^2+t^2}\ ,\qquad \lambda_3(s,t)\equiv \lambda\,.$$
Since $\lambda_3$ is constant, we may suppose that $f$ depends only on $\lambda_1$ and $\lambda_2$. Since $\partial_1f(\lambda_1,\lambda_2)=\partial_2f(\lambda_1,\lambda_2)$ and $\partial_{11}f(\lambda_1,\lambda_2)=\partial_{22}f(\lambda_1,\lambda_2)$, we have the Taylor expansion
$$f\big(\lambda_{1}(s,t),\lambda_{2}(s,t)\big)=f(\lambda_1,\lambda_2)+(t^2+s^2)\partial_{11}f(\lambda_1,\lambda_2)-(t^2+s^2)\partial_{12}f(\lambda_1,\lambda_2)+o(s^2+t^2),$$
which implies $\partial_{a_{12}a_{13}}F(A_0)=0.$\vspace{0.1cm}\\
\noindent\bf Case Mixed.2.3. \rm $\lambda_1=\lambda_2\neq\lambda_3$. Then the eigenvalues are solutions to the equation
$$(x-\lambda_3)\Big((x-\lambda_1)^2-t^2\Big)=s^2(x-\lambda_1)\,.$$
Thus, for every $\eps>0$ we can find $\delta>0$ such that for $s^2+t^2<\delta$, there is a solution $\lambda_3(s,t)$ such that
$$\lambda_3-\eps\le \lambda_3(s,t)\le \lambda_3+\eps\,.$$
Thus, for some $C=C(\lambda_3-\lambda_1,\eps,\delta)$, we obtain
$|\lambda_3(s,t)-\lambda_3|\le Cs^2.$
Substituting this estimate back into the equation, we get
$$\lambda_3(s,t)=\lambda_3+\frac{s^2}{\lambda_3-\lambda_1}+o(s^2+t^2).$$
Similarly, we notice that if $x$ is a solution to one of the equations
\begin{equation}\label{e:ajoajo12}
x-\lambda_1=\frac{s^2+\sqrt{s^4+4t^2(x-\lambda_3)^2}}{2(x-\lambda_3)},\qquad
x-\lambda_1=\frac{s^2-\sqrt{s^4+4t^2(x-\lambda_3)^2}}{2(x-\lambda_3)}
\end{equation}
then $P(x)=0$. It is immediate to check that for every $\eps>0$ there is $\delta>0$ such that for $s^2+t^2<\delta$, there are solutions $\lambda_1(s,t)$ and $\lambda_2(s,t)$ to the two equations in \eqref{e:ajoajo12}, both lying in the interval $(\lambda_1-\eps,\lambda_1+\eps)$. In order to compute the second-order expansions of $\lambda_1(s,t)$ and $\lambda_2(s,t)$, we set for $j=1,2$
$$x_j(s,t):=\lambda_j(s,t)-\lambda_j\ ,\qquad \eta:=\lambda_1-\lambda_3\qquad\text{and}\qquad A^2:=s^4+4t^2\eta^2\ .$$
Thus, the equations in \eqref{e:ajoajo12} become
\begin{equation*}
x_1=\frac{s^2+\sqrt{A^2+8t^2\eta x_1+4t^2x_1^2}}{2(x_1+\eta)}
\qquad\text{and}\qquad
x_2=\frac{s^2-\sqrt{A^2+8t^2\eta x_2+4t^2x_2^2}}{2(x_2+\eta)}\ .
\end{equation*}
We notice that there is some $C=C(\eps,\delta,\eta)$ such that
$$|x_j(s,t)|\le CA\qquad\text{and}\qquad A\le C\sqrt{s^2+t^2}\,.$$
Now, since
$$
\sqrt{A^2+8t^2\eta x_j+4t^2x_j^2}=A\sqrt{1+8t^2A^{-2}\eta x_j+4t^2A^{-2}x_j^2}
=A\Big(1+4t^2A^{-2}\eta x_j\Big)+o(s^2+t^2)
$$
we get that
\begin{align*}
2x_j^2&+2\eta\Big(1-2t^2A^{-1}\Big)x_j-(s^2+A)+o(s^2+t^2)=0
\end{align*}
Then, by exploiting the definition of $A$, we get \begin{align*}
x_1&=\frac{\eta}2\bigg(-1+2t^2A^{-1}+\sqrt{1-4t^2A^{-1}+2\eta^{-2}A+4t^4A^{-2}+2\eta^{-2}s^2}\bigg)+o(s^2+t^2)\\
&=\frac{1}{2\eta}\big(A+s^2\big)+o(s^2+t^2)
\end{align*}
and analogously,
\begin{align*}
x_2&=\frac{\eta}2\bigg(-1-2t^2A^{-1}+\sqrt{1+4t^2A^{-1}-2\eta^{-2}A+4t^4A^{-2}+2\eta^{-2}s^2}\bigg)+o(s^2+t^2)\\
&=\frac{1}{2\eta}\big(-A+s^2\big)+o(s^2+t^2).
\end{align*}
Finally, we get
\begin{align*}
\lambda_1(s,t)&=\lambda_1+\frac{1}{2(\lambda_1-\lambda_3)}\bigg(s^2+\sqrt{s^4+4t^2(\lambda_1-\lambda_3)^2}\bigg)+o(s^2+t^2)\\
\lambda_2(s,t)&=\lambda_2+\frac{1}{2(\lambda_1-\lambda_3)}\bigg(s^2-\sqrt{s^4+4t^2(\lambda_1-\lambda_3)^2}\bigg)+o(s^2+t^2)\\
\lambda_3(s,t)&=\lambda_3+\frac{s^2}{\lambda_3-\lambda_1}+o(s^2+t^2).
\end{align*}
Now, since $\partial_1 f(\lambda_1,\lambda_2,\lambda_3)=\partial_2 f(\lambda_1,\lambda_2,\lambda_3)$, we get
\begin{align*}
f\Big(\lambda_1(s,t),\lambda_2(s,t),\lambda_3(s,t)\Big)&=f(\lambda_1,\lambda_2,\lambda_3)+\frac{s^2}{\lambda_1-\lambda_3}\Big(\partial_1 f(\lambda_1,\lambda_2,\lambda_3)-\partial_3 f(\lambda_1,\lambda_2,\lambda_3)\Big)\\
&\qquad+t^2\Big(\partial_{11} f(\lambda_1,\lambda_2,\lambda_3)-\partial_{12} f(\lambda_1,\lambda_2,\lambda_3)\Big)+o(s^2+t^2),
\end{align*}
which again implies that $\partial_{a_{12}a_{13}}F(A_0)=0.$\\
\noindent\bf Case Mixed.2.4. \rm $\lambda_1=\lambda_3\neq\lambda_2$. Clearly, this case coincides with Case Mixed 2.3, up to a relabeling of the indices.\\
\noindent\bf Case Mixed.2.5. \rm $\lambda_1,\lambda_2$ and $\lambda_3$ are all different.
Then
$$P(x)=(x-\lambda_1)(x-\lambda_2)(x-\lambda_3)-t^2(x-\lambda_3)-s^2(x-\lambda_2)\,.$$
Let $\lambda_{j}(s,t)$, $j=1,2,3$, be the solutions to $P(x)=0$. As above, we notice that for any $\eps>0$ there is $\delta>0$ such that for $s^2+t^2<\delta$, there is a solution $\lambda_1(s,t)$ such that $|\lambda_1(s,t)-\lambda_1|< \eps$. Thus, for some $C=C(\lambda_1,\lambda_2,\lambda_3,\eps,\delta)$, we infer that
$|\lambda_1(s,t)-\lambda_1|\le C(s^2+t^2).$
Substituting this estimate into the equation, we get
\begin{align*}
\lambda_1(s,t)&=\lambda_1+\frac{t^2}{\lambda_1-\lambda_2}+\frac{s^2}{\lambda_1-\lambda_3}+o(s^2+t^2).
\end{align*}
Finally, for $j=2,3$, we reason analogously by rewriting the characteristic equation as
\begin{align*}(x-\lambda_2)\Big((x-\lambda_1)(x-\lambda_3)-s^2\Big)&=t^2(x-\lambda_3)\\
(x-\lambda_3)\Big((x-\lambda_1)(x-\lambda_2)-t^2\Big)&=s^2(x-\lambda_2)\,.\end{align*}
Finally, we get 
$\lambda_2(s,t)=\lambda_2 + t^2 (\lambda_2 -\lambda_1)^{-1}, \lambda_3(s,t)=\lambda_3 + s^2 (\lambda_3 -\lambda_1)^{-1}$ 
and again
\begin{align*}
f\Big(\lambda_1(s,t),\lambda_2(s,t),\lambda_3(s,t)\Big)&=f(\lambda_1,\lambda_2,\lambda_3)+\frac{t^2(\lambda_1-\lambda_3)+s^2(\lambda_1-\lambda_2)}{(\lambda_1-\lambda_3)(\lambda_1-\lambda_2)}\partial_1 f(\lambda_1,\lambda_2,\lambda_3)\\
&\qquad+\frac{t^2}{\lambda_2-\lambda_1}\partial_2 f(\lambda_1,\lambda_2,\lambda_3)+\frac{s^2}{\lambda_3-\lambda_1}\partial_3 f(\lambda_1,\lambda_2,\lambda_3)+o(s^2+t^2),
\end{align*}
leads to $\partial_{a_{12}a_{13}}F(A_0)=0.$
\end{proof}
\begin{proposition}\label{p:second-der-F}
	Let $f:\R^n\to\R$ be a symmetric $C^2$ function, $\Omega\subset\R^n$ be an open set and
 $A \colon \O\to S(n)$ be such that:
 \begin{enumerate}
\item for every pair of indices $1\le i,j\le n$, $a_{ij} \in C^2(\Omega)$;
\item there exists $x_0\in\Omega$ such that $A_0:=A(x_0)$ is diagonal with eigenvalues $$\Lambda(A_0):=\Big(\lambda_1(x_0),\dots,\lambda_n(x_0)\Big)\in \R^n\,.$$
 \end{enumerate}
 Then, the function $F \circ A \colon \O\to \R$ defined as
 \begin{equation}\label{e:definition-of-F(X)}
F(A(x)) :=f\Big(\lambda_1(A(x)),\dots,\lambda_n(A(x))\Big)\,.
\end{equation}
 It is twice differentiable at $x_0$ and
	\begin{align*}\frac{\partial}{\partial x_k}\Big[F(A)\Big](x_0)&=\sum_{1\le i\le n}\partial_if(\Lambda(A_0))\partial_k a_{ii}(x_0),\\
	\frac{\partial^2}{\partial {x_k}\partial x_h}\Big[F(A)\Big](x_0)&=\sum_{1\le i\le n}\partial_{i}f\big(\Lambda(A_0)\big)\partial_{kh} a_{ii}(x_0)\\
	&\qquad+\sum_{1\le i<j\le n}\partial_{a_{ij}}^2 F(A_0)\partial_{k} a_{ij}(x_0)\partial_{h} a_{ij}(x_0)\\
	&\qquad\qquad+\sum_{1\le i\le n}\sum_{1\le j\le n}\partial_{ij}f(\Lambda(A_0))\partial_h a_{ii}(x_0)\partial_k a_{jj}(x_0),
	\end{align*}
	where, for $i\neq j$, we set
	\begin{align*}
	\partial_{a_{ij}}^2F(A_0)
	=\begin{cases}
	\displaystyle 2\frac{\partial_jf\big(\Lambda(A_0)\big)-\partial_if\big(\Lambda(A_0)\big)}{\lambda_j(A_0)-\lambda_i(A_0)}&\quad\text{if}\quad \lambda_i(A_0)\neq\lambda_j(A_0)\,;\medskip\\
2\partial_{ii}f\big(\Lambda(A_0)\big)-2\partial_{ij}f\big(\Lambda(A_0)\big)&\quad\text{if}\quad \lambda_i(A_0)=\lambda_j(A_0)\,.
	\end{cases}
	\end{align*}
\end{proposition}	
\begin{proof}
	By direct computation, for every $x\in\Omega$ we have 	
 $$
 \frac{\partial}{\partial x_k}\Big[F(A)\Big](x)=\sum_{1\le i\le j\le n}\frac{\partial F}{\partial a_{ij}}(A(x))\frac{\partial a_{ij}}{\partial x_k}(x)
 $$
 and
	\begin{align*}
	\frac{\partial^2}{\partial {x_k}\partial x_h}\Big[F(A)\Big](x)  &=\sum_{1\le i\le j\le n}\frac{\partial F}{\partial a_{ij}}(A(x))\frac{\partial^2 a_{ij}}{\partial x_k\partial x_h}(x)\\
 &\qquad+\sum_{1\le i'\le n}\sum_{1\le i\le n}\frac{\partial^2 F(A(x))}{\partial a_{i'i'}\partial a_{ii}}\frac{\partial a_{i'i'}}{\partial x_h}(x)\frac{\partial a_{ii}}{\partial x_k}(x)\\
  &\qquad+\sum_{1\le i< j\le n}\frac{\partial^2 F(A(x))}{\partial a_{ij}^2}\frac{\partial a_{ij}}{\partial x_h}(x)\frac{\partial a_{ij}}{\partial x_k}(x).
	\end{align*}
	Now, the claim follows from Lemma \ref{l:main-lemma-derivatives} since
 $$\sum_{1\le i\le n}\sum_{1\le j\le n}\frac{\partial^2 F(A(x_0))}{\partial a_{ii}\partial a_{jj}}\frac{\partial a_{ii}}{\partial x_h}(x_0)\frac{\partial a_{jj}}{\partial x_k}(x_0) = \sum_{1\le i<j\le n}\partial_{ij}f(\Lambda(A_0)) \partial_k a_{jj}(x_0)\partial_h a_{ii}(x_0).
$$
\end{proof}

\end{document}